\documentclass[12pt]{amsart}

\usepackage{graphicx} 
\usepackage{amsmath, amsfonts, amsthm, amssymb, color, enumerate, extarrows, cite}
\usepackage{mathrsfs, xy,  graphicx, paralist, fancyvrb, makecell, float}
\usepackage{tabularx}
\usepackage{caption}
\usepackage{longtable}
\usepackage{pdflscape}
\usepackage{array}
\usepackage{comment}
\usepackage{tikz}               
\usepackage{tikz-cd}            
\usepackage[a4paper, margin=0.8in]{geometry} 
\usepackage[dvipsnames]{xcolor}
\usepackage[nobysame, alphabetic]{amsrefs}
\usepackage[bookmarks, colorlinks=true, linkcolor=blue, citecolor=blue, urlcolor=blue]{hyperref}
\usepackage{booktabs}
\usepackage{longtable}
\usepackage{array}
\usepackage{ragged2e}

\newcolumntype{L}[1]{>{\RaggedRight\arraybackslash}p{#1}}

\input xy
\xyoption{all}

\usepackage{enumerate, amsmath, amsthm, amsfonts, amssymb, xy,  mathrsfs, graphicx, paralist, fancyvrb}
\usepackage[usenames, dvipsnames]{xcolor}
\usepackage[margin=0.8in]{geometry} 
\usepackage[bookmarks, colorlinks=true, linkcolor=blue, citecolor=blue, urlcolor=blue]{hyperref}

\input xy
\xyoption{all}

\numberwithin{equation}{section}
\newtheorem{theorem}[equation]{Theorem}

\newtheorem{proposition}[equation]{Proposition}
\newtheorem{lemma}[equation]{Lemma}
\newtheorem{corollary}[equation]{Corollary}

\theoremstyle{definition}
\newtheorem{rmk}[equation]{Remark}
\newenvironment{remark}[1][]{\begin{rmk}[#1]\pushQED{\qed}}{\popQED \end{rmk}}
\newtheorem{eg}[equation]{Example}
\newenvironment{example}[1][]{\begin{eg}[#1] \pushQED{\qed}}{\popQED \end{eg}}
\newtheorem{defn}[equation]{Definition}
\newenvironment{definition}[1][]{\begin{defn}[#1]\pushQED{\qed}}{\popQED \end{defn}}

\newenvironment{subeqns}[1][]{\addtocounter{equation}{-1}
\begin{subequations}

}{\end{subequations}}

\makeatletter
\@addtoreset{equation}{section}
\renewcommand{\thesubsection}{%
  \ifnum\c@subsection<1 \@arabic\c@section
  \else \thesection.\@arabic\c@subsection
  \fi
}
\makeatother

\newcommand{\C}{\mathbb{C}}

\newcommand{\cE}{\mathcal{E}}

\newcommand{\rE}{\mathrm{E}}

\newcommand{\cF}{\mathcal{F}}

\newcommand{\sF}{\mathscr{F}}

\newcommand{\rH}{\mathrm{H}}

\newcommand{\cJ}{\mathcal{J}}

\newcommand{\bK}{\mathbf{K}}
\newcommand{\cK}{\mathcal{K}}

\newcommand{\cL}{\mathcal{L}}

\newcommand{\cO}{\mathcal{O}}

\newcommand{\cP}{\mathcal{P}}

\newcommand{\cQ}{\mathcal{Q}}

\newcommand{\cR}{\mathcal{R}}

\newcommand{\bS}{\mathbf{S}}
\newcommand{\cS}{\mathcal{S}}

\newcommand{\cT}{\mathcal{T}}

\newcommand{\ba}{\mathbf{a}}

\newcommand{\fb}{\mathfrak{b}}

\newcommand{\fg}{\mathfrak{g}}

\newcommand{\fh}{\mathfrak{h}}

\newcommand{\bn}{\mathbf{n}}
\newcommand{\fn}{\mathfrak{n}}

\newcommand{\fp}{\mathfrak{p}}

\newcommand{\bt}{\mathbf{t}}

\newcommand{\fu}{\mathfrak{u}}

\renewcommand{\phi}{\varphi}
\renewcommand{\emptyset}{\varnothing}

\renewcommand{\tilde}[1]{\widetilde{#1}}
\newcommand{\ol}[1]{\overline{#1}}

\makeatletter
\def\Ddots{\mathinner{\mkern1mu\raise\p@
\vbox{\kern7\p@\hbox{.}}\mkern2mu
\raise4\p@\hbox{.}\mkern2mu\raise7\p@\hbox{.}\mkern1mu}}
\makeatother

\DeclareMathOperator{\coker}{coker}
\renewcommand{\hom}{\operatorname{Hom}}

\DeclareMathOperator{\rank}{rank}

\DeclareMathOperator{\Tor}{Tor}

\DeclareMathOperator{\Spec}{Spec}

\newcommand{\GL}{\mathbf{GL}}

\newcommand{\fgl}{\mathfrak{gl}}

\newcommand{\btau}{\boldsymbol{\tau}}

\newcommand{\defi}[1]{{\bf \textsf{#1}}}

\theoremstyle{plain}
\newtheorem*{theorem*}{Theorem}
\newtheorem{lem}[equation]{Lemma}
\newtheorem{claim}[equation]{Claim}
\newtheorem*{notation*}{Notation}

\newcommand{\Ber}{\mathcal{B}\mathrm{er}}
\newcommand{\sHom}{\mathcal{H}\mathrm{om}}
\newcommand{\bos}{{\rm bos}}

\title{Cohomology of Complex Supertori}
\date{May 29, 2026}

\author[H. Bhatia]{Hargun Bhatia}
\address{Department of Mathematics, University of California San Diego, La Jolla, CA, USA}
\email{hpbhatia@ucsd.edu}

\author[S. Ganguly]{Soumya Ganguly}
\address{Department of Mathematics,
Rutgers University, New Brunswick, NJ, USA}
\email{soumya.ganguly@rutgers.edu}

\author[Z. Jiang]{Zhiyuan Jiang}
\address{Department of Mathematics, Purdue University, West Lafayette, IN, USA}
\email{jian1206@purdue.edu}

\author[J. Rabin]{Jeffrey M Rabin}
\address{Department of Mathematics, University of California San Diego, La Jolla, CA, USA}
\email{jrabin@ucsd.edu}

\author[S. Sam]{Steven V Sam}
\address{Department of Mathematics, University of California San Diego, La Jolla, CA, USA}
\email{ssam@ucsd.edu}

\begin{document}

\begin{abstract}
We consider supertori which are quotients of affine superspace by translations by algebraically independent odd parameters. Specifically, we describe the ring structure of its space of global sections by generators and relations and completely determine the coherent cohomology groups of its structure sheaf by reducing it to a problem of Lie algebra cohomology. We also show that Poincar\'e duality on sheaf cohomology is compatible with that of the group cohomology of the translation group and give explicit tables of examples. Finally, we compute the Picard groups of these supertori.
\end{abstract}

\maketitle

\tableofcontents

\section{Introduction}

A basic fact in algebraic / complex geometry is that globally defined regular functions on projective varieties are necessarily constant. This can fail in the generalization to super algebraic geometry, and our initial goal in this paper is to study this failure for super analogues of complex tori.

By a complex supertorus we mean a quotient supermanifold of the type $M=\C^{m|n}/G$, where $G \cong {\mathbb Z}^{2m}$ is generated by $2m$ commuting translations on the affine superspace $\C^{m|n}$ such that the underlying space (bosonic reduction) $M_{\bos}$ is a complex $m$-torus. More precisely, we allow for families of such objects over a finitely generated Grassmann algebra $\Lambda$, providing ``constant'' odd parameters $\alpha_{ij}$ by which the odd coordinates can be translated in the $G$-action. 
Supertori are simple and computable examples of complex supermanifolds, and they provide possible models for the Jacobians or Picard groups of super Riemann surfaces, although these objects are not yet fully understood \cite{BergveltRabin1999}.

More generally, it is interesting to consider all of the cohomology groups of the structure sheaf $\rH^\bullet(M, \cO_M)$. This is an algebra over $\rH^0(M,\cO_M)$, and its structure was completely determined when $m=1$ in \cite{RabinRhoadesKim2023}. Our aim is to extend this to general $m$. One of our main results is a complete determination of the dimension of this cohomology (in fact, it has the structure of a representation which we discuss below and we completely determine its irreducible decomposition) and of the ring structure of $\rH^0(M,\cO_M)$ for all $m$. In particular, the first cohomology group $\rH^1(M,\cO_M)$ is of great interest because it gives the Picard group of the supertorus. As an application of our main theorems, in Section~\ref{Section:Picard}, we compute the even Picard group $\operatorname{Pic}_{\rm ev}^0(M)$ in terms of explicit constraints and relations.

By some generalities in homological algebra, we can translate the problem into one of group cohomology. To fix notation, let $z_1,\dots,z_m$ and $\theta_1,\dots,\theta_n$ denote the even and odd coordinates on $\C^{m|n}$. Let $E$ be the exterior algebra with generators
\[
\{\alpha_{ij} \mid 1 \le i \le m,\ 1 \le j \le n\} \cup \{\theta_1,\dots,\theta_n\}.
\]
Then for all $p$, we have an identification
\[
  \rH^p(M,\cO_M) \cong \rH^p(G, N \otimes E),
\]
where $N = \cO_{\C^m}(\C^m)$ denotes the ring of globally defined regular functions on $\C^m$. Here, we can either interpret this algebraically, in which case $N$ is a polynomial ring in $m$ variables, or we can interpret this analytically, in which case $N$ is the ring of globally defined holomorphic functions in $m$ variables. This does not affect the final answer, but the complexity of the next step changes dramatically based on which interpretation is taken.

Since $G$ is a free abelian group, the right side amounts to computing the homology of a Koszul complex, but we are unable to do this computation directly. In degree zero, the cocycles (global functions) are independent of the even coordinates, essentially because the underlying space has only constant global functions. In Sections~\ref{ss:group-coh} and \ref{ss: fourier_proof}, we show that the same is true for all the $\rH^p(M,\cO_M)$. More precisely, there is a subgroup $\cT$ of $G$ generated by $T_1,\dots,T_m$ (acting on $E$ by $T_i(\theta_j) = \theta_j + \alpha_{ij}$ and $T_i(\alpha_{k\ell}) = \alpha_{k\ell}$) such that 
\[
  \rH^p(G,N \otimes E) \cong \rH^p(\cT,E).
\]

Thus the cohomology is essentially algebraic, coming from the odd translations in $G$. For example, a single odd translation $\theta \mapsto \theta + \alpha$ admits an invariant function $\theta \alpha$.

In the algebraic category, this reduction is fairly simple and essentially amounts to the fact that any polynomial satisfying $f(z_1,\dots,z_m) = f(z_1+c_1,\dots,z_m+c_m)$, for all integers $c_i$, must be constant. On the other hand, there are plenty of non-constant holomorphic functions that satisfy this condition, and consequently, our proof for the analytic setting must deal with this issue. By working with Fourier series expansions of such periodic holomorphic functions, we show instead that the group cohomology of $\cT$ with coefficients in the space of non-constant periodic holomorphic functions vanishes. 

This simplifies the Koszul complex substantially, but we can make one further identification. Taking inspiration from \cite{RabinRhoadesKim2023}, the operators $T_i-1$ on $E$ can be identified with the exponentials of commuting nilpotent operators
\[
  X_{i,m+1} = \sum_{k=1}^n \alpha_{ik} \frac{\partial}{\partial \theta_k}.
\]
Hence we have a third Koszul complex using the $X_{i,m+1}$, and in Section \ref{ss:lie-coh}, we show that it is isomorphic to the Koszul complex used to compute $\rH^\bullet(\cT,E)$. A crucial observation is that the action of the operators $X_{i,m+1}$ on $E$ can be extended to an action of the reductive Lie algebra $\fg =\fgl(m+1) \times \fgl(n)$. This allows us to employ the Borel--Weil--Bott theorem (specifically, Kostant's variation for Lie algebra cohomology) to completely determine the groups $\rH^\bullet(M,\cO_M)$  in Section~\ref{ss:irred_decomp} as representations of the subalgebra $\fgl(m) \times \fgl(n)$ (there is no action of $\fg$ on the cohomology).

\begin{theorem}\label{thm:sheafcohschur}
  Fix $0\leq p\leq m$. For a partition $\lambda =(\lambda_1,\ldots,\lambda_{m+1})$, we define
\[ 
    w_p\lambda = (\lambda_1, \dots, \lambda_{m-p}, \underbrace{\lambda_{m-p+2} - 1, \dots, \lambda_{m+1}-1}_{\text{$p$ terms}}, \lambda_{m-p+1} + p).
\]
There is an isomorphism of $\fgl(m) \times \fgl(n)$-representations
\begin{align*}
  \rH^p(M,\cO_M) \cong \bigoplus_\lambda \bS_{((w_p\lambda)_1,\dots,(w_p\lambda)_m)}(\C^m) \otimes \bS_{\lambda^T}(\C^n)
\end{align*}
where $\bS_\lambda$ is the Schur functor of the partition $\lambda$, the sum is over all partitions $\lambda = (\lambda_1,\dots,\lambda_{m+1})$ with $\lambda_1 \le n$, and $\lambda^T$ is the conjugate partition of $\lambda$. The cohomology vanishes in degrees $>p$.
\end{theorem}

There exist several explicit formulas for the dimensions of Schur functors, so from this description one can deduce explicit dimensions of the cohomology groups $\rH^p(M,\cO_M)$ for any given values of $m$ and $n$. However, a general formula not involving sums over partitions is unknown to us and we do not explore this combinatorial aspect any further in this paper (but see Remark~\ref{rmk:m=1} for comments on the case $m=1$).

\medskip

Thus far, the identifications of cohomology above are additive only and we have not discussed the multiplicative structure. 
Going back to our first identification of sheaf cohomology with group cohomology, we show in Section~\ref{ss:duality} that, under these isomorphisms, Serre duality for supertori is compatible with Poincar\'e duality for group cohomology.  

\begin{theorem}\label{thm:duality}
    There is a diagram
       \[
        \begin{tikzcd}
            \rH^i(M,\mathcal{O}_M) \times \rH^{m-i}(M,\mathcal{O}_M)\arrow{d}\arrow{r} & \C \\
            \rH^i(\mathcal T,E) \times \rH^{m-i}(\mathcal T,E)\arrow{ur}
        \end{tikzcd}
      \]
      which commutes up to a nonzero scalar.
\end{theorem}
This is done by first showing that the trace map of Serre duality for the supertorus $M$ is related to the trace map for the bosonic reduction $M_{\bos}$ via a ``top coefficient" morphism $\gamma \colon \cO_M\rightarrow \cO_{M_{\bos}}$ and then checking that these isomorphisms between sheaf cohomology and group cohomology are compatible with cup products and morphisms induced by $\gamma$.

In Appendix~\ref{appendixB}, we also work out Poincar\'e duality for the case $m=n=2$.

Finally, in Section~\ref{section: gen_rel}, we determine a (minimal) set of generators and relations for $\rH^0(M,\mathcal O_M)$ as a $\C$-algebra. Generators can be found explicitly using the representation structure above by computing highest weight vectors for each of the irreducible representations that appear. Similarly, a set of quadratic relations that they satisfy can also be obtained explicitly by direct computation. Proving that these relations are complete requires significantly more effort, but before jumping into details, we first state our main result.

\begin{theorem}\label{thm: ringgeneratorandrelations}
Let $M$ be a supertorus of dimension $m|n$. Then
\begin{itemize}
    \item[(i)] $\rH^0(M,\mathcal O_M)$ is generated as a $\C$-algebra by $\{\alpha_{ij}\}$ and $\{q_N\}$, where 
    \[
        q_N = \nu_1!\dots\nu_n! \sum_{\{i_1,\dots,i_{m},j\} = N} \alpha_{1i_1}\dots\alpha_{mi_m}\theta_{j}
    \]
    for a $(m+1)$-multiset $N = \{1^{\nu_1},\dots,n^{\nu_n}\}$ of $\{1,\dots,n\}$.
    \item[(ii)] The following relations amongst the generators of $\rH^0(M,\mathcal{O}_M)$ generate all relations:
    \begin{itemize}
\item For a fixed $1\leq i\leq m$ and a $(m+2)$-multiset $N$ of $\{1, \dots,n\}$, we have
  \[
    \sum_{j \in N}\alpha_{ij}q_{N \setminus \{j\}}=0.
  \]
\item For fixed multisets $N_1,N_2$ of sizes $m$ and $m+2$, respectively, we have
  \[
    \sum_{j\in N_2}q_{N_1\cup \{j\}}q_{N_2\backslash \{j\}}=0.
  \]
\end{itemize}
\end{itemize}
\end{theorem}
The second set of equations are odd analogues of Pl\"ucker relations (to be made precise shortly).

To show completeness of the relations, we fix $m$ and vary $n$. Then we can think of $\rH^0(M,\cO_M)$ as an algebra in the category of polynomial functors. Crucially, there is a symmetry $\Omega$ of this category that, roughly speaking, transforms anti-commutative algebras into commutative ones. In particular, we get a commutative algebra $A:= \Omega \rH^0(M,\cO_M)$, to which we can apply techniques from algebraic geometry. Specifically, we desingularize the affine variety $\Spec(A)$ by a vector bundle over a Grassmannian (a Kempf collapsing) and show that $\Spec(A)$ has rational singularities.

This puts us in the context of \cite{weyman} and allows us to interpret its Tor groups as sheaf cohomology over this Grassmannian. This computation is complicated in general, but since we are only interested in equations, we only need to understand the first Tor group. We carry out this strategy, which identifies the first Tor group as a representation. Finally, we show that these representations match up with the span of the equations written above. 

\begin{remark}
Under the transformation $\Omega$, the span of the $q_N$ generators is an exterior power, and its elements can be interpreted as the maximal minors of a matrix. Under this interpretation, the second set of equations above translate precisely to Pl\"ucker relations for these minors. In particular, the problem of computing all Tor groups for $A$ contains the computation of all syzygies of Pl\"ucker ideals in general, which is a difficult problem.
\end{remark}

\noindent {\bf Related work.} The sheaf cohomology for the structure sheaf of super analogues of projective space is computed in \cite{MR4544562}. The generalization to Grassmannians is worked out in \cite{superres}, and for the periplectic Grassmannians (these are odd analogues of the isotropic Grassmannians for symplectic and orthogonal groups) in \cite{superres2}. The corresponding problem for super analogues of 2-step partial flag varieties is considered in \cite{abhik}.  The more general problem of computing cohomology of Schur functors of tautological bundles over super Grassmannians is studied in \cite{bwfact}.

A super Riemann surface of genus 1, as studied in \cite{Rabin1995}, is {\it not} a supertorus in our sense, because the group $G$ there acts by superconformal transformations, although the cohomology is the same.

\smallskip
\noindent {\bf Acknowledgments.} We thank Brendon Rhoades for useful discussions. S. Sam was partially supported by NSF grant DMS-2302149.

\section{Background}

\subsection{Supermanifolds}\label{subsec:supgeosuptori}
There are excellent references for the basic material on supergeometry and supermanifolds. To name a few, we have \cites{BergveltRabin1999, Witten2019supmanintegration, ManinGraugefieldtheorycomplexgeo, helein2020introductionsupermanifoldssupersymmetry}. The first reference specifically deals with supercurves in detail. For completeness and to fix notation, we recall the definition of a supermanifold. 

We first recall that a complex \defi{superalgebra} is a complex algebra $R=R_0 \oplus R_1$ such that $R_i \cdot R_j \subseteq R_{i+j}$ (where the indices are understood to be elements of $\mathbb{Z}/2$). Elements of $R_0$ are called \defi{even}, elements of $R_1$ are called \defi{odd}; an element $x$ is \defi{homogeneous} if it is either even or odd, and we let $|x| \in \mathbb{Z}/2$ denote its parity. We say that $R$ is \defi{supercommutative} if, for all homogeneous elements $x,y$, we have $xy = (-1)^{|x||y|} yx$.

Let $\Lambda$ be a Grassmann algebra over $\C$.
One may write $(\ast,\Lambda)$ for the superscheme $\mathrm{Spec}(\Lambda)$, whose underlying topological space is a single point~$\ast$.  

A \defi{complex supermanifold} of dimension $m|n$ over $\Lambda$ is a pair $(M,\mathcal{O}_M)$ consisting of a topological space $M$ and a sheaf $\mathcal{O}_M$ of supercommutative $\Lambda$-algebras on $M$, together with a morphism
\[
(M,\mathcal{O}_M)\;\to\;(\ast,\Lambda),
\]
subject to the following conditions:  
\begin{enumerate}
    \item The bosonic reduction $M_{\bos}:=(M,\mathcal{O}_{M_{\bos}})$ is a complex manifold of complex dimension $m$. Here 
\[
\mathcal{O}_{M_{\bos}} := \mathcal{O}_M/\mathcal{J}_M,
\]
where $\mathcal{J}_M$ is the ideal sheaf generated by the odd elements.  
    \item There exists an open covering $\{U_\alpha\}$ of $M$ and, on each $U_\alpha$, odd sections $\theta^\alpha_1,\dots,\theta^\alpha_n\in \mathcal{O}_M(U_\alpha)$ which are linearly independent over $\mathcal{O}_{M_{\bos}}(U_\alpha)$ and generate an exterior algebra, i.e.,
\[
\mathcal{O}_M(U_\alpha)\;\cong\; \mathcal{O}_{M_{\bos}}(U_\alpha)\otimes_\C\Lambda[\theta^\alpha_1,\dots,\theta^\alpha_n].
\]
Each $U_\alpha$ is called a \defi{coordinate chart}, and coordinates on it are denoted by
\[
(z^\alpha_1,\dots,z^\alpha_m \mid \theta^\alpha_1,\dots,\theta^\alpha_n),
\]
where $(z^\alpha_1,\dots,z^\alpha_m)$ are ordinary holomorphic coordinates on the reduced space.  
\end{enumerate}
In this setup, $m$ is called the \defi{even (Bosonic) dimension}, and $n$ is called the \defi{odd (Fermionic) dimension}.

Let $M$ be a supermanifold with odd dimension $n$. We define the \defi{fermionic sheaf} of $M$ to be the $\cO_{M_{\bos}}$-module 
\[
\cF_M = \mathcal J_M/\mathcal J_M^2.
\]
This is a locally free sheaf of rank $n$.
We also define the line bundle $\Theta = \det(\cF_M)$. A supermanifold $M$ is said to be \defi{projected} if there exists a morphism $p \colon M\to M_{\bos}$ such that $p\circ i = {\rm id}_{M_{\bos}}$ where $i \colon M_{\bos}\to M$ is the bosonic reduction.

On overlaps $U_\alpha\cap U_\beta$, the coordinates transform according to
\[
z^\beta_i \;=\; F^\beta_{\alpha,i}(z^\alpha,\theta^\alpha), \qquad 
\theta^\beta_j \;=\; \Psi^\beta_{\alpha,j}(z^\alpha,\theta^\alpha),
\]
where the $F^\beta_{\alpha,i}$ are even holomorphic functions and the $\Psi^\beta_{\alpha,j}$ are odd ones.  

In this work, we are specifically interested in supermanifolds known as  `complex supertori'. These are the super analogues of the standard complex tori from complex geometry (more of this standard theory can be found in \cite{GrifiththsHarrisAlgebraicgeo1978, Mumfordabelianvarieties1970}).

\subsection{The Berezinian}

In this section, we discuss the Berezinian sheaf for complex supermanifolds. Let $M$ be a complex supermanifold of dimension $m|n$, and let $\{(U_\alpha; z_1,\dots,z_m \mid \theta_1,\dots,\theta_n)\}$ be an atlas. 

In a coordinate chart,  the integral of a function is defined as the ordinary integral of the coefficient of $\theta_1\cdots\theta_n$. The Berezinian is the line bundle making this local definition invariant under coordinate changes, see \cite[Definition 4.1]{NojaForms}.

\begin{definition}
     The \defi{Berezinian} sheaf $\Ber_M$ is the right $\cO_M$-module characterized by the following transition rule: On each chart $U_\alpha$, it is freely generated by a section $\mathcal D_\alpha$ as a right $\cO_M$-module such that 
    \[
        \mathcal D_\beta|_{U_\alpha\cap U_\beta} = \mathcal D_\alpha|_{U_\alpha\cap U_\beta}\cdot \operatorname{Ber}(\operatorname{Jac}(\phi_{\alpha\beta}))
    \]
    where $\phi_{\alpha\beta}$ is the transition map from $U_\alpha$ to $U_\beta$, with
    \[
         \operatorname{Jac}(\varphi_{\alpha\beta}) = \begin{pmatrix}
            \partial_z \varphi_{\alpha\beta,\rm ev} & \partial_\theta \varphi_{\alpha\beta,\rm ev} \\
            \partial_z \varphi_{\alpha\beta,\rm odd} & \partial_\theta \varphi_{\alpha\beta,\rm odd}
        \end{pmatrix} =: \begin{pmatrix}
            A & B\\
            C & D
        \end{pmatrix}
    \]
    and 
    \[
        \operatorname{Ber}(\operatorname{Jac}(\varphi_{\alpha\beta})) = \det(A-BD^{-1}C)\det(D)^{-1}. \qedhere
    \]
\end{definition}

As the integration depends only on the coefficient of the top product of $\theta$'s, we have the following adjunction formula \cite[Theorem 5.3]{NojaForms}:

\begin{theorem}\label{adjunction}
    There is an isomorphism of $\cO_{M_{\bos}}$-modules
    \[
        \phi \colon \mathcal J_M^n\Ber_M\to \Omega_{M_{\bos}}^m,\quad\quad \mathcal D_\alpha\cdot \theta_1\cdots\theta_n \cdot f\mapsto dz_1 \wedge \cdots \wedge dz_m\cdot f_{\bos}
    \]
    where $f_{\bos}$ is the image of $f$ under the quotient morphism $\cO_M\to \cO_{M_{\bos}}$.
\end{theorem}

\subsection{Supertori}\label{ss:supertori}

Fix a Grassmann algebra $\Lambda$ which is freely generated by symbols $\alpha_{ij}$ with $i = 1,\dots,m$ and $j = 1,\dots,n$. We denote $\C^{m|n} = \C^{m|n}_\Lambda$ as the affine superspace of dimension $m|n$ over $\Lambda$. If we regard the $\Lambda$-supermanifold $M$ as a $\C$-supermanifold, its odd dimension becomes \[ N := mn+n, \] so that its dimension is $m|N$.

Let $G$ be the free abelian group generated by $S_1,\dots,S_m,T_1,\dots,T_m$. We will denote by $\mathcal S$ (resp. $\mathcal T$) the subgroup of $G$ generated by $S_i$ (resp. $T_i$).  

For $1\leq i,j\leq m$, pick $\tau_{ij}\in \C$ so that the determinant of the imaginary part of the (normalized) period matrix is non-zero, i.e., $\det({\rm Im}(\tau_{ij}))\neq 0$. Define vectors $\btau_i = (\tau_{i1}, \ldots, \tau_{im})$. Consider the following $G$-action on the supermanifold $\C^{m|n}$ (considered as a locally ringed space): 
\begin{align*}
    S_i(z_1, \ldots, z_m, \theta_1, \ldots, \theta_n) &= (z_1, \ldots, z_i +1, \ldots, z_m, \theta_1, \ldots, \theta_n)\\
    T_i(z_1, \ldots, z_m, \theta_1, \ldots, \theta_n) &= (z_1+\tau_{i1}, \ldots, z_m+\tau_{im}, \theta_1 + \alpha_{i1}, \ldots, \theta_n + \alpha_{in})
\end{align*}
for $i = 1,\dots,m$.
We define the \defi{supertorus} $M$ of dimension $m|n$ over $\Lambda$ to be the quotient
\[
M:=\C^{m|n}/G,
\]
together with the quotient morphism
\[
\pi\colon \C^{m|n}\to M.
\]
From the construction, we have $\cO_{M} = \pi_*\cO_{\mathbb C^{m|n}}^G$. For details on the quotient construction, see \cite{MR2667819, quotient-supermanifold}.

\begin{lemma}
    The supertorus $M$ is a projected supermanifold.
\end{lemma}

\begin{proof}
    The inclusion $\cO_{\C^m}\to \cO_{\C^{m|n}}$ induces a morphism $p \colon \cO_{M_{\bos}} = \pi_*\cO_{\C^m}^G\to \pi_*\cO_{\C^{m|n}}^G = \cO_{M}$ satisfying $p\circ i = {\rm id}_{M_{\bos}}$. 
\end{proof}

Next, we construct an isomorphism of right $\cO_M$-modules from $\cO_M$ to $\Ber_M$.

Let $M$ be the supertorus of dimension $m|n$. For each chart $U_\alpha$,  we assign an isomorphism of right $\cO_M$-modules $\psi_\alpha \colon \cO_{M}|_{U_\alpha}\to \mathcal D_\alpha\cdot \cO_{U_\alpha}$ given by $1\mapsto\mathcal D_\alpha$.

\begin{lem}\label{BerTrivial}
    The local isomorphisms $\psi_\alpha$ glue together to an isomorphism of right $\cO_M$-modules $\psi \colon \mathcal O_M\to \Ber_M$, i.e., $\mathcal Ber_M$ is trivial.
\end{lem}

\begin{proof}
    Let $\{V_\alpha\subseteq \C^{m|n}\}_\alpha$ be a collection of open balls such that for any $\alpha,\beta$, we have $V_\alpha\cap gV_\beta\neq \emptyset$ for at most one $g\in G$. Then $\{U_\alpha = \pi(V_\alpha)\}_{\alpha}$ is an atlas of $M$.   For intersecting charts $U_\alpha,U_\beta$, the transition map is given by
    \begin{align*}
        \varphi_{\alpha\beta}\colon  U_\alpha|_{U_\alpha\cap U_\beta} &\to U_\beta|_{U_\alpha\cap U_\beta}\\
        z_k^\alpha|\theta_l^\alpha &\mapsto z_k^\alpha+t_k|\theta_l^\alpha+\gamma_l
    \end{align*}
    where $t_k = 0,1$ or some $\tau_{ij}$, and $\gamma_l = 0$ or some $\alpha_{ij}$ are both constants in $\Lambda$.
    
    As a result, we have
    \[
        \operatorname{Jac}(\varphi_{\alpha\beta}) = \begin{pmatrix}
            \partial_z \varphi_{\alpha\beta,\rm ev} & \partial_\theta \varphi_{\alpha\beta,\rm ev} \\
            \partial_z \varphi_{\alpha\beta,\rm odd} & \partial_\theta \varphi_{\alpha\beta,\rm odd}
        \end{pmatrix} = \begin{pmatrix}
            I_m & 0\\
            0 & I_n
        \end{pmatrix}.
    \]
    This implies $\operatorname{Ber}(\operatorname{Jac}(\varphi_{\alpha\beta})) = 1$. Thus the isomorphisms $\psi_\alpha \colon \cO_{M}|_{U_\alpha}\to \mathcal D_\alpha\cdot \cO_{U_\alpha}$ are compatible with the transition functions so that they glue to a global isomorphism $\psi$ uniquely. 
\end{proof}

\begin{remark}
    The Berezinian $\Ber_M$ is canonically defined, while the isomorphism $\Ber_M\cong \cO_M$ constructed above depends on the choice of the covering $\pi \colon \C^{m|n}\to M$.
\end{remark}

\subsection{Highest weight theory}

We recall some relevant notions from Lie theory; all terminology not defined explicitly here is standard and can be found, for example, in \cite{Humphreys72Liealgbook}.
Let $\fg$ be a finite dimensional semisimple complex Lie algebra and $V$ be a finite dimensional $\fg$-representation. We fix a Cartan subalgebra $\fh\subseteq \fg$. Then the $\fh$-representation induces a decomposition
\[
    V = \bigoplus_{\lambda\in \fh^*}V_\lambda
\]
where $V_\lambda = \{v\in V \mid h\cdot v = \lambda(h)v \text{ for all } h\in\fh\}$. We call $V_\lambda$ a \defi{weight space} whose elements are \defi{weight vectors} and we call $\lambda\in\fh^*$ a \defi{weight}.

Let $\Phi$ be the root system of $\fg$. We choose a base $\Delta$ of it and define the \defi{Borel subalgebra} as 
\[
    \mathfrak{b} = \fh\oplus\bigoplus_{\alpha\in\Phi^{+}}\fg_\alpha.
\]
We define the \defi{nilpotent radical} of $\mathfrak{b}$ as $\mathfrak{n} = [\mathfrak b,\mathfrak b]$. We call a weight vector $v\in V$ a \defi{highest weight vector} (with respect to $(\fg,\fh,\Delta)$) if $x\cdot v = 0$ for all $x\in \fn$. The weight corresponding to this vector is called a \defi{highest weight}. If $V$ is generated by a highest weight vector as a $\fg$-representation, we say $V$ is a \defi{highest weight representation}.

If $V$ is irreducible, then there exists a unique highest weight and the collection of all highest weight vectors is a $1$-dimensional subspace of $V$.

The theory continues to work when $\fg$ is a reductive Lie algebra, i.e., a product of a semisimple Lie algebra with an abelian Lie algebra, as long as we assume that the abelian part acts by diagonalizable operators. The latter is automatic if $\fg$ is the Lie algebra of a complex Lie (or algebraic) group $G$ where the abelian part is the Lie algebra of a torus and the representation comes from $G$.

Now we specialize to the case $\fg = \fgl_n(\C)$. We can take $\fh$ to be the subalgebra of diagonal matrices and $\fb$ to be the subalgebra of upper-triangular matrices. In that case, $\fn$ is the subalgebra of strictly upper-triangular matrices. We identify weights with elements of $\C^n$. 

A \defi{partition} of a nonnegative integer $n$ is a tuple of nonnegative integers $\lambda = (\lambda_1, \ldots, \lambda_r)$ such that $\lambda_1 \ge \cdots \ge \lambda_r$ and $\lambda_1+\dots+\lambda_r = n$.  The \defi{length} of $\lambda$, denoted by $\ell(\lambda)$, is the number of nonzero entries of $\lambda$. We will generally regard two partitions as equivalent if their subsequence of positive entries are identical; in this way, for each $r$, we can speak interchangeably about partitions of length $\le r$ and nonnegative sequences of integers of length $r$, and we will do so without further comment below.

We use \defi{Young diagrams} to represent partitions graphically: for a partition $\lambda$, its Young diagram is a left-justified array of boxes with $\lambda_i$ boxes in row $i$. The \defi{conjugate partition} $\lambda^T$ to the partition $\lambda$ is defined by reflecting the Young diagram of $\lambda$ across the main diagonal. More explicitly, we have $\lambda^T_i = \#\{j \mid \lambda_j \ge i\}$.

For each partition $\lambda$, we have a corresponding Schur functor $\bS_\lambda$ (defined on vector spaces) which satisfies $\bS_\lambda(V) \ne 0$ if and only if $\dim V \ge \ell(\lambda)$. Detailed expositions of this topic can be found in  \cite{FultonHarrisreptheory1991, weyman}. In particular, $\bS_\lambda(\C^n)$ is a representation of $\fgl_n(\C)$; if $\ell(\lambda) \le n$, then this is an irreducible representation with highest weight $\lambda$. There are two familiar cases: if $\lambda = (d)$, then $\bS_\lambda$ is the $d$th symmetric power, while if $\lambda = (1,\dots,1)$, then $\bS_\lambda$ is the $d$th exterior power.

\section{Cohomology of supertori}

Our goal is to compute the coherent cohomology of $M$, i.e., the sheaf cohomology of its structure sheaf.

\subsection{Group cohomology description of $\rH^i$} \label{ss:group-coh}

Let $E$ be the exterior algebra with generators
\[
\{\alpha_{ij} \mid 1 \le i \le m,\ 1 \le j \le n\} \cup \{\theta_1,\dots,\theta_n\}.
\]
and  $N=\mathcal O_{\C^{m|n}}(\C^{m|n})$ be the ring of global regular functions. Then we have
\[
    N=\mathcal{O}_{\C^m}(\C^m)\otimes_{\C} E,
\]

Recall that $G$ is the free abelian group generated by $S_1,\dots,S_m,T_1,\dots,T_m$ and we denote $\mathcal S$ (resp. $\mathcal T$) as the subgroup of $G$ generated by $S_i$ (resp. $T_i$).  Then $G$, and hence also $\mathcal S$ and $\mathcal T$, act on both $E$ and $N$.

First, we need the following lemma.

\begin{lem}
    $\rH^i(M, \mathcal{O}_M)\cong \rH^i(G, N)$ for all $i$.
\end{lem}

\begin{proof} 
First, since $\mathcal O_{\C^{m|n}}\cong \mathcal O_{\C ^m}\otimes_\C E$ as sheaves of abelian groups, we conclude that 
\[
\rH^i(\C^{m|n},\mathcal O_{\C^{m|n}}) \cong \rH^i(\C ^m,\mathcal O_{\C ^m}\otimes E) = 0 \text{ for $i>0$.}
\]

    Hence we can apply \cite[Appendix to \S 2, Property (c)]{Mumfordabelianvarieties1970} to the covering map on the underlying spaces $\pi\colon \C^m\to |M|$ (here $|M|$ is the underlying topological space of $M$) and the sheaf $\mathcal O_M$ to get an isomorphism of abelian groups
    \[
        \rH^i(G,\pi^{-1} \mathcal O_M(\C ^m))\to \rH^i(M,\mathcal O_M).
    \]
    Since $\pi$ is a covering map, we have $\pi^{-1} \mathcal O_M \cong \mathcal O_{\C^{m|n}}$ as sheaves of abelian groups from the definition of $\mathcal O_M$.
\end{proof}

We first show that the cohomology is independent of the even coordinates $z$, in the sense that every class in $\rH^i(M,\mathcal O_M)$ has a representative valued in $E$. 

\begin{proposition}  \label{prop:HM=HTE}
    $\rH^i(M,\mathcal O_M) \cong \rH^i(\mathcal T,E)$ for all $i$.
\end{proposition}

\begin{proof}
Since $\mathcal{S}$ and $\mathcal{T}$ commute with each other, we have $G/\mathcal{S}\cong \mathcal{T}$. Then the Hochschild--Serre spectral sequence \cite[6.8.2]{Weibel1994} applied to the following short exact sequence
\[
    0\rightarrow \mathcal{S}\rightarrow G\rightarrow \mathcal{T}\rightarrow 0
\] 
yields
\[
    \rE^{pq}_{2} = \rH^p(\mathcal{T}, \rH^q(\mathcal{S},N)) \Longrightarrow \rH^{p+q}(G, N).
\]
The proof now follows from two claims which are proven below:  $\rH^i(\mathcal S,N)=0$ for $i>0$ (Lemma~\ref{lem:SNpos-vanish}) and $\rH^i(\cT, \rH^0(\mathcal S,N))=\rH^i(\cT,E)$ for all $i\geq 0$ (Lemma~\ref{indepofz}). 
\end{proof}

\begin{notation*}
To avoid conflicting with indices, we will write $\sqrt{-1}$ throughout rather than the commonly used symbol $i$.
\end{notation*}

First, we deal with the easier lemma. 

\begin{lem} \label{lem:SNpos-vanish}
    $\rH^i(\mathcal S,N)=0$ for $i>0$. 
\end{lem}

\begin{proof}
    Since $\mathcal S$ acts trivially on $E$, we have 
    \[
    \rH^i(\mathcal S,N)=\rH^i(\mathcal S, \mathcal{O}(\C^m))\otimes_{\C} E.
    \]
    So, it suffices to show that $\rH^i(\mathcal S, \mathcal{O}(\C^m))=0$ for all $i>0$. Since $\rH^i(\C^m, \mathcal{O}_{\C^m})=0$ for $i>0$, we can apply \cite[Appendix to \S 2, Property (c)]{Mumfordabelianvarieties1970}, to get an isomorphism 
    \[
    \rH^i(\mathcal{S}, \mathcal{O} (\C^m))\to \rH^i(\C^m/\mathcal{S}, \mathcal{O}_{\C^m/\mathcal{S}}).
    \]
    Note that
\[
\C^m/\mathcal{S}\cong (\C^*)^m
\]
via the map induced by the coordinatewise exponential
\[
(z_1,\dots,z_m)\longmapsto \bigl(e^{2\pi \sqrt{-1} z_1},\dots,e^{2\pi \sqrt{-1} z_m}\bigr).
\]
Next, \((\C^*)^m\) is a Stein manifold by \cite[Chapter~V, \S 1, Theorem~5]{GrauertSteinspacebook79} since we can write:
\[
(\C^*)^m=\C^m\setminus H,
\qquad
H=\{z_1\cdots z_m=0\}.
\]
So  $\rH^i(\C^m/\mathcal{S}, \mathcal{O}_{\C^m/\mathcal{S}})=0$ for $i>0$. Thus, we get $\rH^i(\mathcal{S}, \mathcal{O} (\C^m))=0$ for all $i>0$,  as desired.
\end{proof}

\subsection{Proof that $\rH^i(\cT,\rH^0(\mathcal S,N))=\rH^i(\cT,E)$}  \label{ss: fourier_proof}

As before, since $\mathcal S$ acts trivially on $E$, we have $\rH^0(\mathcal S,N)=\rH^0(\mathcal S, \mathcal{O}(\C^m))\otimes_{\C} E$. 
Our goal then is to prove the following lemma.

\begin{lem}\label{indepofz}
$\rH^i(\mathcal T,\rH^0(\mathcal S, \mathcal{O}(\C^m))\otimes_{\C} E) \cong \rH^i(\mathcal T, E)$ for all $i\geq 0$.
\end{lem}

\begin{proof}
We first note that $\rH^0(\mathcal{S}, \mathcal{O}(\C^m))$ consists of periodic holomorphic functions on $\C^m$. Any such function has a (unique) Fourier expansion
\[
    f(\mathbf{z})=\sum_{\mathbf{n}\in \mathbb{Z}^m}f_{\mathbf{n}}\exp(2\pi \sqrt{-1} \mathbf{n}\cdot \mathbf{z})
\]
To see this, we recall that $\C^m/\mathcal{S} \cong (\C^*)^m$. Under this identification, the Fourier expansion above is nothing but a Laurent expansion, where such a unique series representation is well known \cite[Page~35, Theorem~2]{Shabatcomplexsev92}. We observe that $T_i$ acts on $\rH^0(\mathcal{S}, \mathcal{O}(\C^m))$ by 
\[
T_if(\mathbf{z})=f(\mathbf{z}+\btau_i)
=\sum_{\mathbf{n}\in \mathbb{Z}^m} f_{\mathbf{n}}\exp(2\pi \sqrt{-1} \mathbf{n} \cdot \btau_i)\exp(2\pi \sqrt{-1} \mathbf{n} \cdot \mathbf{z}).
\]
For any $\mathbf{n}\neq \mathbf{0}$ we note that as the matrix $({\rm Im}(\btau))$ is non-singular, we have $({\rm Im}(\btau)) \mathbf{n}\neq \mathbf{0},$ which means there exists $1\leq i\leq m$ such that ${\rm Im}(\btau_i) \cdot \mathbf{n} \neq 0$. We define 
\[
\Psi\colon  \mathbb{Z}^m\backslash \{\mathbf{0}\}\to \{1,2,\ldots,m\}, \qquad 
    \Psi(\mathbf{n}):=\underset{1 \le i \le m}{\operatorname{argmax}}\ |{\rm Im}(\btau_i) \cdot \mathbf{n}|
\]
If there is more than one argmax, we choose the smallest index. Now, for $1\leq i\leq m$, we define
\[
\mathcal{P}_i:= \{ f\in \rH^0(\mathcal{S}, \mathcal{O}(\C^m)) \mid f(\mathbf{z})=
\sum_{\substack{{\mathbf{n}\in \mathbb{Z}^m\backslash\{\mathbf{0}\}}\\ \Psi(\mathbf{n})=i}}f_{\mathbf{n}}\exp(2\pi \sqrt{-1} \mathbf{n}\cdot\mathbf{z}) \}
\]
to be those functions whose Fourier coefficients $f_\bn$ vanish when $\mathbf{n}=\mathbf{0}$ or when $\Psi(\mathbf{n})\neq i$. We also define
\[
\mathcal{P}_0:=  \{ f\in \rH^0(\mathcal{S}, \mathcal{O}(\C^m)) \mid f(\mathbf{z})=f_\mathbf{0}\}
\] 
to be the constant functions. It's clear from the action of $T_i$ that these are $\mathcal{T}$-submodules of $\rH^0(\mathcal{S}, \mathcal{O}(\C^m))$ and we have a direct sum decomposition as $\mathcal{T}$-modules:
\[
 \rH^0(\mathcal{S}, \mathcal{O}(\C^m))=\bigoplus_{0\leq i\leq m} \mathcal{P}_i.
\]

Since $\mathcal{P}_0 \cong \C$ is a trivial $\cT$-module, we have
\[
\rH^j(\mathcal{T}, \rH^0(\mathcal{S}, \mathcal{O}(\C^m))\otimes_{\C} E)=\rH^j(\mathcal{T},E)\oplus \bigoplus_{i= 1}^m \rH^j(\mathcal{T}, \mathcal{P}_i\otimes_{\C}E)
\]
for any $j\geq 0$. Thus, to prove the lemma, it suffices to show that $\rH^j(\mathcal{T}, \mathcal{P}_i\otimes_{\C}E)=0$ for all $1\leq i\leq m$ and $j\geq 0$. This will be a consequence of the following claims. 

\begin{claim}
For each $1\leq j\leq m$, the operator $T_j-1$ is invertible on $\mathcal{P}_j$.
\end{claim}

\addtocounter{equation}{-1}
\begin{subequations}
\begin{proof}
    Let us fix a $j \in \{1, \ldots, m\}$. For $f \in \mathcal{P}_j$,  we have
    \[
    (T_j -1)f = \sum_{\substack{{\mathbf{n}\in \mathbb{Z}^m\backslash\{\mathbf{0}\}}\\ \Psi(\mathbf{n})=j}} (\exp(2 \pi \sqrt{-1} \mathbf{n} \cdot  \btau_j) - 1) f_{\mathbf{n}}\exp(2\pi \sqrt{-1} \mathbf{n} \cdot \mathbf{z}).
    \]
    By orthogonality of the Fourier basis, if $(T_j -1) f =0$ we have $(\exp(2\pi \sqrt{-1} \mathbf{n} \cdot  \btau_j) - 1)f_{\mathbf{n}} =0$ for all $\mathbf{n} \in \mathbb{Z}^m\backslash\{\mathbf{0}\}$ such that $\Psi(\mathbf{n}) = j$. By definition of $\Psi$, we have $\mathbf{n} \cdot  \btau_j \ne 0$ and in fact, because $({\rm Im}(\btau_j)) \cdot \mathbf{n} \ne 0$, we have that $\mathbf{n} \cdot  \btau_j \notin \mathbb{Z}$, so $\exp(2\pi \sqrt{-1} \mathbf{n} \cdot  \btau_j) - 1 \ne 0$. Hence we have $f_{\mathbf{n}} =0$ for all $\mathbf{n} \in \mathbb{Z}^m\backslash\{\mathbf{0}\}$ such that $\Psi(\mathbf{n}) = j$. So $f = 0$ and the map $(T_j -1)$, on $\mathcal P_j$, is injective. 

    To prove surjectivity, pick $F \in \mathcal P_j$ and let $F_\bn$ be its Fourier coefficients. 
Define the series
\begin{equation}\label{eq:surjtjminus1series}
    \tilde{F} = \sum_{\substack{{\mathbf{n}\in \mathbb{Z}^m\backslash\{\mathbf{0}\}}\\ \Psi(\mathbf{n})=j}}
    \frac{F_{\mathbf{n}}}{\exp(2 \pi \sqrt{-1} \mathbf{n} \cdot \btau_j) - 1} \exp(2\pi \sqrt{-1} \mathbf{n} \cdot \mathbf{z}).
    \end{equation}
    We will show that this defines an element of $\cP_j$; once this is done, we will have $(T_j-1)\tilde{F} = F$ by definition.

We now do some basic estimates. We know from Section \ref{subsec:supgeosuptori} that $\operatorname{Im}(\btau)=(\operatorname{Im}(\tau_{jk}))_{j,k=1}^m$ is invertible. Then its smallest singular value, $\sigma_{\min}(\operatorname{Im}(\btau))$, is positive, i.e. 
\[
\sigma_{\min}(\operatorname{Im}(\btau))=\min_{\|x\|=1}\|\operatorname{Im}(\btau) x\|>0.
\]
For any nonzero integer vector $\mathbf{n}\in\mathbb{Z}^m$ we have $\|\bn\|\ge1$, hence
\[
\|\operatorname{Im}(\btau) \mathbf{n}\|=\|\mathbf{n}\|\,\Big\|\operatorname{Im}(\btau)\!\Big(\tfrac{\mathbf{n}}{\|\mathbf{n}\|}\Big)\Big\|\;\ge\;\|\mathbf{n}\|\,\sigma_{\min}(\operatorname{Im}(\btau))\;\ge\;\sigma_{\min}(\operatorname{Im}(\btau)).
\]
Writing the coordinates of $\operatorname{Im}(\btau) \mathbf{n}$ as
\[
\operatorname{Im}(\btau_{\ell}) \cdot  \mathbf{n}=\sum_{k=1}^m n_k\,\operatorname{Im}(\tau_{\ell k}),
\]
we use the inequality $\max_{\ell} |\operatorname{Im}(\btau_{\ell}) \cdot \mathbf{n}|\ge \|\operatorname{Im}(\btau) \cdot \mathbf{n}\|/\sqrt{m}$ to obtain
\[
\max_{1\le \ell \le m}\Big|\sum_{k=1}^m n_k\,\operatorname{Im}(\tau_{\ell k})\Big|\;\ge\;\frac{\sigma_{\min}(\operatorname{Im}(\btau))}{\sqrt{m}}.
\]
Since $\Psi(\mathbf{n}) = j$, we have 
\begin{align}\label{ineq:lowerbdindepn}
    \Big|\sum_{k=1}^m n_k\,\operatorname{Im}(\tau_{jk})\Big|\;\ge\;\frac{\sigma_{\min}(\operatorname{Im}(\btau))}{\sqrt{m}}>0,
\end{align}
which provides a lower bound independent of $\bn$. Now, 
\[
    \begin{aligned}
        |\exp{(2 \pi \sqrt{-1} \mathbf{n} \cdot \btau_j}) -1| &= |\exp{(2 \pi \sqrt{-1} \mathbf{n} \cdot \operatorname{Re}(\btau_j))}\exp{(-2 \pi  \mathbf{n} \cdot \operatorname{Im}(\btau_j))} -1| \\
        & \geq |1 - \exp{(-2 \pi  \mathbf{n} \cdot  \operatorname{Im}(\btau_j))}| \\
        & \geq 1 -  \exp{(-2 \pi  |\mathbf{n} \cdot \operatorname{Im}(\btau_j)|)}
        \\
        & \geq 1 -  \exp{(-2 \pi  \sigma_{\min}(\operatorname{Im}(\btau))/\sqrt{m})} > 0.
    \end{aligned}
\]
The first inequality follows from the triangle inequality $|a-b|\geq ||a|-|b||$. 
The second inequality follows because for any real $x$, we have $|1 - \exp{(-x)}| \geq 1 - \exp{(-|x|)}$. 
The third inequality follows from \eqref{ineq:lowerbdindepn}.
 
We observe that the series 
\begin{align}\label{Fexpression}
    F=\sum_{\substack{{\mathbf{n}\in \mathbb{Z}^m\backslash\{\mathbf{0}\}}\\ \Psi(\mathbf{n})=j}} F_{\mathbf{n}}\exp(2\pi \sqrt{-1} \mathbf{n} \cdot\mathbf{z})
\end{align}
converges absolutely and uniformly on any compact subset of $\C^m/\mathcal{S}$. This can be done again by the same identification $\C^m/\mathcal{S} \cong (\C^*)^m$, as before. Under this identification, the Fourier expansion above is nothing but a Laurent expansion whose convergence is well known \cite[Page~35, Theorem~2]{Shabatcomplexsev92}. After a suitable change of coordinates on $(\C^*)^m$, such a Laurent series can be written as a finite sum of power series, each converging normally (for a definition of normal convergence, see \cite[\S 2.2]{HormanderSCV}). Now the estimates above show that for any compact subset $K \subset (\C^*)^m$, where we denote the coordinate of $(\C^*)^m$ as $\mathbf{q}$, we have
\begin{align*}
    & \sup_{\mathbf{q} \in K} \sum_{\substack{{\mathbf{n}\in \mathbb{Z}^m\backslash\{\mathbf{0}\}}\\\Psi(\mathbf{n})=j \\ \|\mathbf{n}\|_1>k}} 
    \left|\frac{F_{\mathbf{n}}}{\exp(2 \pi \sqrt{-1} \mathbf{n} \cdot \btau_j) - 1} \mathbf{q}^{\mathbf{n}} \right|  
    \leq \frac{1}{1 -  \exp(\frac{-2 \pi}{\sqrt{m}}  \sigma_{\min}(\operatorname{Im}(\btau)))} \sup_{\mathbf{q} \in K}  \sum_{\substack{{\mathbf{n}\in \mathbb{Z}^m\backslash\{\mathbf{0}\}}\\\Psi(\mathbf{n})=j\\ \|\mathbf{n}\|_1 > k}}\left|F_{\mathbf{n}}\mathbf{q}^{\mathbf{n}} \right|.
\end{align*}
This approaches 0 as $k \to \infty$ by the normal convergence of the series \eqref{Fexpression}. By \cite[Corollary 2.2.4]{HormanderSCV}, the series $\tilde{F}$ defined by \eqref{eq:surjtjminus1series} 
defines a function in $\rH^0(\mathcal{S}, \mathcal{O}(\C^m))$ as it is holomorphic and clearly periodic; moreover, this function, because of its series expansion, is in $\mathcal P_j$.
\end{proof}
\end{subequations}

\begin{claim}
Let $\cP$ be any $\mathcal{T}$-module over $\C$ such that $T_i-1$ acts invertibly on $\cP$. Then $T_i-1$ also acts invertibly on $\cP\otimes_{\C} E$.
\end{claim}

\begin{proof}
The action of $T_i-1$ on $\cP \otimes_\C E$ is given by
\[
T_i \otimes T_i - 1 \otimes 1 = (T_i-1) \otimes T_i + 1 \otimes (T_i - 1).
\]
The first term is invertible by assumption, and the second term is nilpotent. Since the two terms commute, the sum is also invertible.
\end{proof}

From the above two claims, we conclude that $T_i-1$ acts invertibly on $\mathcal{P}_i\otimes_{\C}E$ for any $1\leq i\leq m$. 
Now, let $A=\C[\mathcal{T}]\cong \C[T_1^{\pm},\ldots, T_m^{\pm}]$ be the group ring of $\mathcal{T}$. First, $\mathbf{t}:= (T_1-1,T_2-1,\ldots, T_m-1)$ is a regular sequence  in $A$. Recall that this means that $T_1-1$ is a nonzerodivisor in $A$ and, for each $i=2,\dots,m$, the coset of $T_i-1$ is a nonzerodivisor in $A/(T_1-1,\dots,T_{i-1}-1)$; this clearly holds as the latter is isomorphic to $\C[T_i^{\pm}, \dots, T_m^\pm]$. 

Since $A/(T_1-1,\ldots,T_m-1)\cong \C$, we see that the Koszul complex $K(\mathbf t)$ provides a free resolution of $\C$ (readers are referred to \cite[Section~4.5]{Weibel1994} for details).

Thus, the required group cohomology is the homology of $\hom_A(K(\bt), \cP_i \otimes E)$. However, we have the decomposition
\begin{align*}
\hom_A(K(\bt), \cP_i \otimes E) &\cong \left(\bigotimes_{j=1}^m (A \xrightarrow{\cdot (T_j-1)} A) \right)\otimes_A (\cP_i \otimes E)\\
&\cong \left(\bigotimes_{\substack{j=1\\j\neq i}}^{m} (A \xrightarrow{\cdot (T_j-1)} A)\right) \otimes_A (\cP_i \otimes E \xrightarrow{\cdot (T_i-1)} \cP_i \otimes E)
\end{align*}
By invertibility, the homology of $(\cP_i \otimes E \xrightarrow{\cdot (T_i-1)} \cP_i \otimes E)$ is 0, so the result follows from basic properties of tensor products of complexes \cite[Acyclic Assembly Lemma 2.7.3]{Weibel1994}.
\end{proof}

\begin{remark}
If we want to treat $M$ as an algebraic supervariety, then the ring of global functions on $\C^m$ consists of polynomials in $m$ variables and the proof becomes much simpler (there are no non-constant periodic polynomials!) and we sketch it here.
 We want to show that $\rH^0(\cS, \C[z_1,\dots,z_m]) = \C$ and $\rH^i(\cS, \C[z_1,\dots,z_m]) = 0$ for $i>0$.

  First, suppose that $m=1$. Then the group ring $\C[\cS] = \C[S_1^\pm]$ is a Laurent polynomial ring in 1 variable (hence a PID), and so we have the following free resolution for the trivial module $\C$:
  \[
    0\to \C[S_1^\pm] \xrightarrow{\cdot (S_1-1)} \C[S_1^\pm] 
  \]
  where the differential means ``multiply by $S_1-1$''. Now we apply $\hom_{\C[S_1]}(-, \C[z_1])$ to get
  \begin{align*} \label{eqn:m=1complex}
    \C[z_1] \xrightarrow{\phi} \C[z_1], \qquad \phi(f) = f(z_1+1) - f(z_1)
  \end{align*}
  where the left term has cohomological degree 0 and the right term has cohomological degree 1. For $n \ge 0$, define $v_n = \frac{1}{n!} z_1(z_1-1)(z_1-2) \cdots (z_1-(n-1))$. Then $\phi(v_0) = 0$ and $\phi(v_n) = v_{n-1}$ for $n > 0$, so we see that
  \[
    \rH^0(\cS, \C[z_1]) = \ker \phi = \C, \qquad \rH^1(\cS, \C[z_1]) = \coker \phi = 0.
  \]
  For general $m$, we take $m$ copies of the above complex, one for each variable $S_i$, and tensor them together over $\C$. Then by the K\"unneth formula
  \cite[Theorem 3.6.3]{Weibel1994}, we get the desired result.
\end{remark}

\subsection{Description in terms of Lie algebra cohomology} \label{ss:lie-coh}

Let $U,V,U'$ be $\C$-vector spaces of dimensions $m,n,1$ respectively. Pick bases $u_1,\ldots, u_m$ for $U$, $v_1,\ldots,v_n$ for $V$ and $u_{m+1}$ for $U'$.
We have the exterior algebra:
\[
E = \bigwedge^{\bullet}((U\oplus U')\otimes V).
\]
This is the same $E$ as in Section~\ref{ss:group-coh}; in our previous notations, $\alpha_{ij}=u_i\otimes v_j$ and $\theta_i= u_{m+1}\otimes v_j$.  For the sake of uniformity, we will interpret $\alpha_{m+1,j} = \theta_j$ below. It has commuting actions of the groups $\GL(U\oplus U'), \GL(V)$ (the reader may think of these as either algebraic groups or complex Lie groups, it will make no difference) and also of their Lie algebras $\fgl(U \oplus U'), \fgl(V)$ defined in the usual way.

Let us make this explicit for $\fgl(U \oplus U')$. Pick $1 \le i \le m+1$ and $1 \le j \le n$. With respect to our basis, let $x_{i,j}$ be the matrix unit with a 1 in row $i$ and column $j$ and 0's elsewhere. Then $x_{i,j}$ acts on $E$ by the operator
\[
  X_{i,j} = \sum_{k=1}^n \alpha_{ik} \frac{\partial}{\partial \alpha_{jk}}.
\]
This is the usual action of $\fgl(U \oplus U')$ on $(U \oplus U') \otimes V \cong (U \oplus U')^{\oplus n}$ and the usual way to extend Lie algebra actions to tensor / exterior / symmetric products is $x \cdot (a \otimes b) = xa \otimes b + a \otimes xb$ which $X_{i,j}$ satisfies.

We can think of matrices in $\fgl(U \oplus U')$ as block matrices (with block sizes $m,1$). Let $\fp$ be the block upper-triangular matrices and let $\fn$ be the strictly block upper-triangular matrices. Then $\fp$ is a parabolic subalgebra (i.e., contains a Borel subalgebra) and $\fn$ is its nilradical. In this case, $\fn$ is the span of $x_{i,m+1}$ where $i=1,\dots,m$. 

Given pairwise commuting operators $\ba = (a_1,\dots,a_n)$ on $E$, the Koszul complex $\bK(\ba)^\bullet$ for $E$ can be written as $\bK(\ba)^k = \bigwedge^k (\C^n) \otimes E$. Let $e_1,\dots,e_n$ be the standard basis for $\C^n$; for $I = (i_1,\dots,i_k)$, we let $e_I = e_{i_1} \wedge \cdots \wedge e_{i_k}$. The differential $d$ is given by
\[
e_I  \otimes \alpha \mapsto \sum_{j \notin I} e_{(I,j)} \otimes a_j \alpha.
\]

\begin{proposition} \label{prop:koszul-isom}
For each $i\in \{1,\ldots,m\}$, there exists an invertible operator $\phi_i$ on $E$ such that $T_i - 1 = X_{i,m+1} \phi_i$. Furthermore, the $\phi_i$ commute with each other and with $T_j-1$ and $X_{j,m+1}$.

In particular, the Koszul complex on $E$ with respect to $T_1-1,\dots,T_m-1$ is isomorphic to the Koszul complex on $E$ with respect to $X_{1,m+1}, \dots, X_{m,m+1}$. 
\end{proposition}

\begin{proof}
As observed in \cite[Proof of Proposition 3.2]{RabinRhoadesKim2023}, we have 
\[
T_i - 1 = \sum_{k \ge 1} \frac{X_{i,m+1}^k}{k!} = X_{i,m+1} \left(1 + \sum_{k \ge 1} \frac{X_{i,m+1}^k}{(k+1)!} \right).
\]
Here the infinite sums are actually finite since $X_{i,m+1}$ is nilpotent on $E$. In particular, $\sum_{k \ge 1} \frac{X^k_{i,m+1}}{(k+1)!}$ is also nilpotent, and we define $\phi_i$ to be this sum plus 1. Commutativity follows since the $X_{j,m+1}$ pairwise commute.

For the second statement, we define an isomorphism 
\[
\Phi \colon \bK(T_1-1,\dots,T_n-1)^\bullet \to \bK(X_{1,m}, \dots, X_{m,m+1})^\bullet, \quad e_I \otimes \alpha \mapsto e_I \otimes (\prod_{j \notin I} \phi_j )\alpha.
\]
To see that $\Phi$ is a chain map:
\begin{align*}
\Phi( d ( e_I \otimes \alpha)) &= \Phi (\sum_{j \notin I} e_{(I,j)} \otimes (T_j-1)\alpha)\\
&= \sum_{j \notin I} e_{(I,j)} \otimes (\prod_{k \notin (I,j)} \phi_k )(T_j-1)  \alpha\\
&= \sum_{j \notin I} e_{(I,j)} \otimes X_{j,m+1} (\prod_{k \notin I} \phi_k) \alpha\\
&= d (e_I \otimes (\prod_{k \notin I} \phi_k) \alpha)\\
&= d (\Phi(e_I \otimes \alpha)) . \qedhere
\end{align*}
\end{proof}

\begin{corollary}\label{Algebraic-Description}
    $\rH^i(M,\cO_M) \cong \rH^i(\mathcal{T},E)\cong \rH^i(\fn, E)$.
\end{corollary}

\begin{proof}
We have
\begin{align*}
\rH^i(\bK(T_1-1,\dots,T_n-1)^\bullet) &= \rH^i(\cT, E),\\
\rH^i(\bK(X_{1,m+1}, \dots, X_{m,m+1})^\bullet) &= \rH^i(\fn, E);
\end{align*}
the second description comes from the Chevalley--Eilenberg complex since $\fn$ is abelian.
\end{proof}

\subsection{Irreducible decomposition of cohomology}\label{ss:irred_decomp}

Fix $0\leq p\leq m$. For a partition $\lambda =(\lambda_1,\ldots,\lambda_{m+1}) \in\mathbb{Z}_{\geq 0}^{m+1}$, we define
\begin{equation*}
    w_p\lambda = (\lambda_1, \dots, \lambda_{m-p}, \underbrace{\lambda_{m-p+2} - 1, \dots, \lambda_{m+1}-1}_{\text{$p$ terms}}, \lambda_{m-p+1} + p).
\end{equation*}
In words: we move the term $\lambda_{m-p+1}$ to the end, add $p$ to it, and subtract $1$ from all terms that we moved it past.

\begin{theorem}\label{Irreducible-Decomposition}
There is an isomorphism  
\begin{align*}
  \rH^p(M,\mathcal{O}_M) \cong \bigoplus_\lambda \bS_{((w_p\lambda)_1,\dots,(w_p\lambda)_m)}(U)\otimes \bS_{(w_p\lambda)_{m+1}}(U') \otimes \bS_{\lambda^T}(V)
\end{align*}
as $\mathfrak{gl}(U)\times \mathfrak{gl}(U')$-modules, where the sum is over all partitions $\lambda = (\lambda_1,\dots,\lambda_{m+1})$ with $\lambda_1 \le n$.
\end{theorem}

\begin{proof}
The Cauchy identity (see, for instance, \cite[Corollary 2.3.3]{weyman}) gives a $\fgl(U \oplus U') \times \fgl(V)$ equivariant isomorphism of $E$ into Schur functors
\[
  E \cong \bigoplus_\lambda \bS_\lambda(U \oplus U') \otimes \bS_{\lambda^T}(V).
\]
Since $\fn \subseteq \fgl(U\oplus U')$ acts trivially on $\bS_{\lambda^T}(V)$, we get
\[
    \rH^p(\fn,E)\cong \bigoplus_{\lambda}\rH^p(\fn,\bS_{\lambda}(U\oplus U')) \otimes \bS_{\lambda^T}(V)
\]
Since $\bS_{\lambda}(U\oplus U')$ is an irreducible representation of $\fgl(U\oplus U')$, we can use Kostant's theorem to compute the Lie algebra cohomology. The relevant combinatorics is worked out in \cite{kostantrho}. More precisely, \cite[Theorem 5.1]{kostantrho} contains the statement for Lie algebra homology for $\fn_- \cong \fn^*$ which is equivalent to Lie algebra cohomology for $\fn$, while \cite[\S 6.1, Lemma 6.3]{kostantrho} specializes to our situation.
\end{proof}

\begin{remark}
Although Kostant's theorem is generally stated for semisimple Lie algebras, we can extend its statement to $\mathfrak{gl}_n$: all of the representations we deal with come from the group ${\bf GL}_n$, so the center acts semisimply. Keeping track of the action of the center allows us to identify weights as elements of $\mathbb{Z}^n$ rather than cosets of the subgroup spanned by $(1,\dots,1)$. Now, the center acts trivially on the subalgebra $\mathfrak{n}$, and hence acts by a constant on the Koszul complex of $\mathfrak{n}$ tensored with $\bS_\lambda(U)$. Hence, for Kostant's theorem, we can restrict to $\mathfrak{sl}_n$ to cite the literature, but then note that all weights appearing in cohomology must have the same size as the input weight $\lambda$. 
\end{remark}

\begin{remark} \label{rmk:m=1}
Specializing to the case $m=1$, we see that the theorem above gives (for $p=0,1$)
\[\text{dim}\; \rH^p(M, \cO_M) = \sum_{\lambda} \text{dim}\; \bS_{\lambda^{T}}(V)
\] where the sum is over all partitions $\lambda = (i,j)$ with $j\leq i\leq n$. Let $s_{\lambda}(x_1,x_2,\ldots, x_n)$ denote the character of $\GL_n(\C)$-representation $\bS_{\lambda}(\C^n)$ and $e_k(x_1,x_2,\ldots, x_ n)$ denote the character of $\GL_n(\C)$-representation $\Lambda^k(\C^n)$. By the Jacobi--Trudi identity \cite[Corollary 7.16.2]{EnumComb2} applied to the partition $\lambda^{T}$, we get
\[
s_{\lambda^{T}} = \text{det} (e_{\lambda_k-k+l})_{1\leq k,l\leq 2} = \text{det} \begin{pmatrix}
            e_{\lambda_1} & e_{\lambda_1+1}\\
            e_{\lambda_2-1} & e_{\lambda_2}
        \end{pmatrix} = e_ie_j-e_{i+1}e_{j-1}.
\]
Thus, we get 
\[
\text{dim}\; \bS_{\lambda^T}(\C^n)= s_{\lambda^T}(1^n)=
\binom{n}{i}\binom{n}{j}-\binom{n}{i+1}\binom{n}{j-1}
\] recovering the dimension results in \cite[Theorem 3.3, Theorem 3.7]{RabinRhoadesKim2023}.
\end{remark}

\subsection{Duality}\label{ss:duality}

Since $M$ is a compact supermanifold, by \cite{HW87}, we have Serre duality 
\begin{equation*}
    \rH^p(M,\mathcal{O}_M^{\vee}\otimes \Ber_M )\otimes_{\C} \rH^{m-p}(M,\mathcal{O}_M) \rightarrow \rH^m(M,\Ber_M)\xrightarrow[]{t} \C
\end{equation*}
giving rise to an isomorphism (using triviality of $\Ber_M$)
\begin{equation*}
\rH^p(M,\mathcal{O}_M) \cong \rH^{m-p}(M,\mathcal{O}_M)^\vee.
\end{equation*}

On the other hand, we have Poincar\'e duality on group cohomology.  The main result of this section is that these two notions of duality are compatible with each other.

\begin{theorem} \label{thm:duality-compat}
    There is a diagram
       \[
        \begin{tikzcd}
            \rH^i(M,\mathcal{O}_M) \times \rH^{m-i}(M,\mathcal{O}_M)\arrow{d}\arrow{r} & \C \\
            \rH^i(\mathcal{T},E) \times \rH^{m-i}(\mathcal{T},E)\arrow{ur}
        \end{tikzcd}
    \] 
    which commutes up to a non-zero scalar.
\end{theorem}

To prove the theorem, we need some preparations. We first construct a morphism of sheaves of abelian groups $\gamma \colon \cO_M\to \cO_{M_{\bos}}$ that extracts the top coefficient of the product of all odd variables, informally called the \defi{top coefficient}. In the following, we will regard $M$ as a supermanifold over $\C$, whose odd dimension is $N:=mn+n$, so that its dimension is $m|N$.

First, consider the morphism
\[
    \ol{\gamma} \colon \pi_*\cO_{\C^{m|n}}\to \pi_*\mathcal J^N_{\C^{m|n}},\quad f = \sum_\epsilon \theta^{\epsilon_1}_1 \cdots \theta^{\epsilon_N}_N\cdot f^{\epsilon_1,\dots,\epsilon_N}\mapsto \theta_1\cdots \theta_N\cdot f^{1,\dots,1}
\]
where $\pi \colon \C^{m|n}\to M$ is the covering map and the sum is over all choices of  $\epsilon_i = 0,1$. 
We claim that it descends to a morphism $\cO_M \to \mathcal J_M^N$.

\begin{lem}\label{lem:top_coeff}
    There is a commutative diagram
    \[
        \begin{tikzcd}
            \cO_M\arrow{r}\arrow{d} & \mathcal J_M^N\arrow{d}\\
            \pi_*\cO_{\C^{m|n}}\arrow{r}{\ol{\gamma}} & \pi_*\mathcal J_{\C^{m|n}}^N
        \end{tikzcd}.
    \]
\end{lem}

\begin{proof}
    Since $\mathcal O_M = \pi_*\cO_{\C^{m|n}}^G$, the vertical morphisms are inclusions of $G$-invariant elements. Assume $f = \sum_\epsilon \theta^{\epsilon_1}_1 \cdots\theta^{\epsilon_N}_N \cdot f^{\epsilon_1,\dots,\epsilon_N}$ is a $G$-invariant element. Note that under the $G$-action, the top coefficient of $g\cdot f$ is $g\cdot f^{1,\dots,1}$, since $\theta_1\cdots\theta_N$ is $G$-invariant. This shows that the top horizontal morphism is well-defined.
\end{proof}

We will define $\gamma \colon \cO_M\to \mathcal J_M^N\to \cO_{M_{\bos}}$ as the composition of the projection with the adjunction morphism in Theorem~\ref{adjunction}. In local coordinates, we have
\[
    \gamma(\sum_\epsilon \theta^{\epsilon_1}_1 \cdots \theta^{\epsilon_N}_N \cdot f^{\epsilon_1,\dots,\epsilon_N}) = f^{1,\dots,1}(z_1,\dots,z_m).
\]

Let $M$ be a projected supermanifold with projection $p \colon M\rightarrow M_{\bos}$. In the following proposition and lemma, we will write $\cO_M$ in place of $p
_*\cO_M$ everywhere (as sheaves of abelian groups, these are the same since $p$ is the identity map of topological spaces, however, $p_*\cO_M$ has a $\cO_{M_{\bos}}$-structure). We will view $\cO_M$ as a $\cO_{M_{\bos}}$-module this way. Note that the morphism $\cO_M\rightarrow \cJ_M^N$ is $\cO_{M_{\bos}}$-linear (and so is $\gamma)$.
\begin{proposition}\label{Ber-supertorus}
    Let $M$ be the complex supertorus with projection $p \colon M\rightarrow M_{\bos}$. 
    \begin{enumerate}
        \item There are isomorphisms of $\cO_M$-modules
        \[
            \Ber_M\cong p^*(\omega_{M_{\bos}}\otimes_{\cO_{M_{\bos}}} \Theta^\vee)
            \cong \sHom_{\cO_{M_{\bos}}}(\cO_M, \omega_{M_{\bos}}).
        \]

        \item There is a morphism of $\cO_{M_{\bos}}$-modules
        \[
            \tilde{\gamma} \colon \Ber_M\cong \sHom_{\cO_{M_{\bos}}}(\cO_M, \omega_{M_{\bos}})
            \to \omega_{M_{\bos}}
        \]
        which on any open $U\subseteq M_{\bos}$ is given by
        \[
             \operatorname{Hom}_{\cO_{M_{\bos}}|_U}(\cO_M|_{U}, \omega_{M_{\bos}}|_{U})
            \rightarrow \omega_{M_{\bos}}(U),
            \quad
            f \mapsto f(U)(1)
        \]
        where for $f \colon \cO_M|_U\to \omega_{M_{\bos}}|_U$, denote by
        $f(U)\colon \cO_M(U)\to \omega_{M_{\bos}}(U)$ the corresponding morphism
        between the sections over $U$.

        \item The trace map of $M$ factors through the trace map of $M_{\bos}$:
        if $(\omega_{M_{\bos}},t)$ is a dualizing sheaf for $M_{\bos}$, then
        $(\Ber_M,t\circ \rH^m(\tilde{\gamma}))$ is a dualizing sheaf for $M$.
    \end{enumerate}
\end{proposition}

\begin{subeqns}
\begin{proof}
    \noindent (1) The first isomorphism holds for any projected supermanifold by \cite[Theorem 1.1]{Noja-Thesis}.
    For the second isomorphism, we note that
    \[
        p^*(\omega_{M_{\bos}}\otimes_{\cO_{M_{\bos}}} \Theta^\vee)
        \cong
        \cO_M \otimes_{\cO_{M_{\bos}}} \omega_{M_{\bos}}
        \otimes_{\cO_{M_{\bos}}} \Theta^{\vee}.
    \]
    Since there is an isomorphism
    \[
        \sHom_{\cO_{M_{\bos}}}(\cO_M, \omega_{M_{\bos}})
        \cong
        \sHom_{\cO_{M_{\bos}}}(\cO_M, \cO_{M_{\bos}})
        \otimes_{\cO_{M_{\bos}}}\omega_{M_{\bos}},
    \]
    it suffices to show that $\cO_M \otimes_{\cO_{M_{\bos}}}\Theta^{\vee}        \cong \sHom_{\cO_{M_{\bos}}}(\cO_M, \cO_{M_{\bos}})$.

    Note that there is an isomorphism $\Theta \cong \cJ_M^N$ between $\cO_{M_{\bos}}$-modules as
    mentioned in \cite[Page 5]{Noja2026poincaredualitysupergravity}. It follows
    that the pairing
    \[
        \cO_M \otimes_{\cO_{M_{\bos}}} \cO_M \rightarrow \cO_M \rightarrow \cJ_M^N\cong \Theta
    \]
    is a perfect pairing, since it is locally given by the perfect pairing $\Lambda^i \cF_M \otimes \Lambda^{N-i}\cF_M\rightarrow \text{det}(\cF_M)$.
    
    Thus, we get the following isomorphisms
    \[
        \cO_M
        \cong
        \sHom_{\cO_{M_{\bos}}}(\cO_M, \Theta)
        \cong
        \sHom_{\cO_{M_{\bos}}}(\cO_M, \cO_{M_{\bos}})
        \otimes_{\cO_{M_{\bos}}} \Theta
    \]
    which imply our result.

    \medskip

    \noindent (2) The projection morphism $\cO_{M_{\bos}}\to \cO_M$ induces the desired    morphism $\tilde{\gamma}$.
    
    \medskip

    \noindent(3) It suffices to show that for any coherent $\cO_M$-module $\mathscr F$,
    \begin{equation}\label{pairing}
        \operatorname{Hom}_{\cO_{M}}(\mathscr F,\Ber_M)\times \rH^m(M,\mathscr F)
        \to \rH^m(M,\Ber_M) \to \mathbb C
    \end{equation}
    is a nondegenerate pairing, where the last map is $t\circ \rH^m(\tilde{\gamma})$.

    Using the isomorphisms in (1) and tensor-hom adjunction, for any coherent
    $\cO_M$-module $\sF$, we have an isomorphism of abelian groups

    \begin{displaymath}
    \begin{aligned}
        \operatorname{Hom}_{\cO_M}(\mathscr F,\Ber_M)
        &\cong \operatorname{Hom}_{\cO_M}
        (\mathscr F,\sHom_{\cO_{M_{\bos}}}(\cO_M, \omega_{M_{\bos}}))\\
        &\cong \operatorname{Hom}_{\cO_{M_{\bos}}}
        (\mathscr F\otimes_{\cO_M} \cO_M,\omega_{M_{\bos}}) \\
        &\cong \operatorname{Hom}_{\cO_{M_{\bos}}}
        (\mathscr F,\omega_{M_{\bos}})
    \end{aligned}
    \end{displaymath}
    and this isomorphism is precisely induced by the morphism $\tilde{\gamma}$.

    Thus, given any $g\in \operatorname{Hom}_{\cO_M}(\mathscr F,\Ber_M)$ and
    $r \in \rH^m(M,\mathscr F)$, the pairing above \eqref{pairing} is
    \[
        t\circ \rH^m(\tilde{\gamma})\circ \rH^m(g)(r)
        =
        t \circ \rH^m(\tilde{\gamma}\circ g)(r)
    \]
    which is precisely the pairing
    \[
        \operatorname{Hom}_{\cO_{M_{\bos}}}(\mathscr F,\omega_{M_{\bos}})
        \times \rH^m(M,\mathscr F)
        \to \rH^m(M,\mathcal \omega_{M_{\bos}}) \to \mathbb C.
    \]
    This is known to be non-degenerate given that $(\omega_{M_{\bos}},t)$ is a
    dualizing sheaf for $M_{\bos}$.
\end{proof}
\end{subeqns}

\begin{lem}\label{trace}
    There is a commutative diagram
    \[
        \begin{tikzcd}
            \rH^m(M,\cO_M)\arrow{r}{t}\arrow{d}[swap]{\rH^m(\gamma)} & \C \\
            \rH^m(M_{\bos},\cO_{M_{\bos}})\arrow{ru}[swap]{t}
        \end{tikzcd}
    \]
    where the $t$'s are the trace maps of Serre duality.
\end{lem}

\begin{proof}
    By Proposition~\ref{Ber-supertorus}, we have an isomorphism of
    $\cO_M$-modules 
    \[
        \Ber_M\cong p^*(\omega_{M_{\bos}}\otimes_{\cO_{M_{\bos}}}\Theta^\vee).
    \]

    We observe that the composition
    \[
        \Ber_M
        \cong
        \cO_M \otimes \Theta^{\vee} \otimes \omega_{M_{\bos}}
        \rightarrow
        \Theta \otimes \Theta^{\vee}\otimes \omega_{M_{\bos}}
        \cong
        \omega_{M_{\bos}},
    \]
    where the morphism $\cO_M\rightarrow \Theta$ is the top coefficient map as defined in Lemma~\ref{lem:top_coeff},
    coincides with the map $\tilde{\gamma}$ in Proposition~\ref{Ber-supertorus}(2).
    
From the proof of \cite[Theorem 1.1]{Noja-Thesis}, we see that the isomorphism \[\Ber_M\cong \cO_M\otimes \Theta^{\vee}\otimes \omega_{M_{\bos}}\] in local coordinates is given by
\[
\mathcal D_{\alpha} \mapsto 1 \otimes (\theta_1\cdots \theta_N)^{\vee} \otimes dz_1\wedge \cdots \wedge dz_m
\]
From this explicit description and definition of $\gamma \colon \cO_M\rightarrow \cO_{M_{\bos}}$ after Lemma ~\ref{lem:top_coeff}, we then obtain a commutative diagram of $\cO_{M_{\bos}}$-modules 
    \[
        \begin{tikzcd}
            \cO_M\arrow{r}{\gamma}\arrow{d} & \cO_{M_{\bos}}\arrow{d} \\
            \Ber_M\arrow{r}{\tilde{\gamma}} & \omega_{M_{\bos}}
        \end{tikzcd}
    \]
    whose vertical maps are isomorphisms. By Proposition~\ref{Ber-supertorus}(3), we are done.
\end{proof}

\begin{proof}[Proof of Theorem~\ref{thm:duality-compat}]
Since the isomorphism $\rH^i(M,\cO_M)\cong \rH^i(\cT, E)$
is given by a sequence of isomorphisms as described in
Section~\ref{ss:group-coh}, and the Poincar\'e duality pairing is given by
cup products followed by the ``top coefficient'' map, we break the triangle
into the following diagram, which we then proceed to show is commutative.
\[
\begin{tikzcd}[
column sep=1.5em,
row sep=2.2em,
font=\small
]
	{\rH^i(M,\mathcal{O}_M) \times \rH^{m-i}(M,\mathcal{O}_M)} & {\rH^m(M,\mathcal{O}_M)} & {\rH^{m}(M_{\bos},\mathcal{O}_{M_{\bos}})} & {\C} \\
	\\
	{\rH^i(G,N) \times \rH^{m-i}(G,N)} & {\rH^m(G,N)} & {\rH^m(G,\mathcal{O}(\C^m))} & {\C} \\
	\\
	{\rH^i(\mathcal{T},\rH^0(\mathcal{S},N)) \times \rH^{m-i}(\mathcal{T},\rH^0(\mathcal{S},N))} & {\rH^m(\mathcal{T},\rH^0(\mathcal{S},N)) } & {\rH^m(\mathcal{T},\rH^0(\mathcal{S}, \mathcal{O}(\C^m)))} & {\C} \\
	\\
	{\rH^i(\mathcal{T},E) \times \rH^{m-i}(\mathcal{T},E)} & {\rH^m(\mathcal{T},E)} & {\rH^m(\mathcal{T},\C)} & {\C}
	\arrow["\cup", from=1-1, to=1-2]
	\arrow["\cong"', from=1-1, to=3-1]
	\arrow["", from=1-2, to=1-3]
	\arrow["\cong"', from=1-2, to=3-2]
	\arrow["{t\; \cong}", from=1-3, to=1-4]
	\arrow["\cong"', from=1-3, to=3-3]
	\arrow["\text{id}"', from=1-4, to=3-4]
	\arrow["\cup", from=3-1, to=3-2]
	\arrow["\cong"', from=3-1, to=5-1]
	\arrow["", from=3-2, to=3-3]
	\arrow["\cong"', from=3-2, to=5-2]
	\arrow["\cong", from=3-3, to=3-4]
	\arrow["\cong"', from=3-3, to=5-3]
	\arrow["\text{id}"', from=3-4, to=5-4]
	\arrow["\cup", from=5-1, to=5-2]
	\arrow["\cong"', from=5-1, to=7-1]
	\arrow["", from=5-2, to=5-3]
	\arrow["\cong"', from=5-2, to=7-2]
	\arrow["\cong", from=5-3, to=5-4]
	\arrow["\cong"', from=5-3, to=7-3]
	\arrow["\text{id}"', from=5-4, to=7-4]
	\arrow["\cup", from=7-1, to=7-2]
	\arrow["", from=7-2, to=7-3]
	\arrow["\cong", from=7-3, to=7-4]
\end{tikzcd}
\]
The maps in the above diagram between the second and third columns are induced
from $\gamma$ and are defined precisely in the proof below.

\begin{itemize}
    \item[I.] For the rectangle made up by the first two rows, we can show
    commutativity (up to a non-zero scalar) in the following steps:
    \begin{itemize}
        \item From Lemma~\ref{trace}, the map
        \[
            \rH^m(M,\mathcal{O}_M)\xrightarrow[]{\rH^m(\gamma)}
            \rH^m(M,\mathcal{O}_{M_{\bos}})            \xrightarrow{t}            \C
        \]
        is the trace map on $\rH^m(M,\mathcal{O}_M)$.

        \item Since Mumford's isomorphism is compatible with cup products
        \cite[Appendix to \S 2, Property (b)]{Mumfordabelianvarieties1970}, we get that the
        leftmost square commutes (cup product is induced by multiplication on $\mathcal{O}_M$ and $N$).

        \item Define the map
        \[
            \partial_{\theta}\colon E \rightarrow \C
        \]
        which extracts the coefficient of the maximal product
        $\prod_{i,j}{\alpha_{ij}}\prod_{i}\theta_i$ (product taken in
        lexicographical order). It is easy to see that the map
        \[
            \gamma(M)\colon
            N=\mathcal{O}(\C^m)\otimes \Lambda[\theta_1,\ldots,\theta_n]
            \rightarrow
            \mathcal{O}(\C^m)
        \]
        induced by the sheaf map
        $\gamma\colon \mathcal{O}_M\rightarrow \mathcal{O}_{M_{\bos}}$
        is $G$-equivariant and is the same as
        $\text{id}\otimes \partial_{\theta}$. Thus, from the naturality of
        Mumford's isomorphism \cite[Appendix to \S2]{Mumfordabelianvarieties1970},
        the middle square commutes.

        \item Since all the morphisms in the rightmost square are
        isomorphisms, we see that it commutes up to a non-zero scalar.
    \end{itemize}

    \item[II.] For the commutativity of the rectangle made up by the second and
    third rows, note that the isomorphisms
    \[
        \rH^i(\mathcal{T}, \rH^0(\mathcal{S},N))\cong \rH^i(G,N),
        \quad
        \rH^m(\mathcal{T}, \rH^0(\mathcal{S}, \mathcal{O}(\C^m)))
        \cong
        \rH^m(G, \mathcal{O}(\C^m))
    \]
    coming from the degeneration of the Hochschild--Serre spectral sequence
    are precisely the inflation map in group cohomology
    \cite[Chapter XI Proposition 10.2]{MacLaneHomology}, which is functorial and
    commutes with cup products
    \cite[Proposition I.5.2 and I.5.3]{CohofNumberFields}.

    \item[III.] For the commutativity of the rectangle made up of the third and
    fourth rows, note that by Lemma~\ref{indepofz}, the isomorphisms
    \[
        \rH^i(\mathcal T, E)\cong \rH^i(\mathcal T,\rH^0(\mathcal S, N)),
        \quad
        \rH^i(\mathcal T, \C)\cong
        \rH^i(\mathcal T,\rH^0(\mathcal S, \mathcal{O}(\C^m)))
    \]
    are induced by the inclusion of algebras
    \[
        E\hookrightarrow \rH^0(\mathcal{S}, \mathcal{O}(\C^m))\otimes E
        = \rH^0(\mathcal{S},N),
        \quad
        \C\hookrightarrow \rH^0(\mathcal{S}, \mathcal{O}(\C^m))
    \]
    and thus the required commutativity follows from functoriality of
    cup product in group cohomology, along with the obvious fact that the
    inclusion is compatible with
    \[
        E\xrightarrow[]{\partial_{\theta}}\C,
        \quad
        \rH^0(\mathcal{S},N)=
        \rH^0(\mathcal{S}, \mathcal{O}(\C^m))\otimes E
        \xrightarrow[]{\text{id}\otimes \partial_{\theta}}
        \rH^0(\mathcal{S}, \mathcal{O}(\C^m)). \qedhere
    \]
\end{itemize}
\end{proof}

\section{The ring structure of $\rH^0(M,\mathcal O_M)$}\label{section: gen_rel}

In this section, we will focus on the $\C$-algebra structure of  $\rH^0(M,\mathcal O_M)$ and give a presentation by generators and relations.
We continue to use the notation from Section~\ref{ss:lie-coh}. So we have
\[
    E = \bigwedge^\bullet ((U\oplus U')\otimes V),
\]
where $U,U',V$ are vector spaces with bases $\{u_1,\dots,u_m\}$, $\{u_{m+1}\}$, and $\{v_1,\dots,v_n\}$, respectively. Note that $E$ has a multi-grading where 
\[
E_{i_1,i_2,\ldots, i_{m+1}}:= \bigwedge^{i_1} (\C u_1 \otimes V) \otimes  \bigwedge^{i_2}(\C u_2\otimes V) \otimes \cdots \otimes \bigwedge^{i_{m+1}}(\C u_{m+1}\otimes V),
\] 
where $i_1, i_2, \ldots, i_{m+1} \in \mathbb{Z}_{\geq 0}^{m+1}$. By Corollary \ref{Algebraic-Description}, $\rH^0(M,\mathcal O_M)$ has the following algebraic characterization:
\[
    \rH^0(M,\cO_M) = \rH^0(\mathfrak{n},E)
\]
as $\C$-subspaces sitting inside $E$, where $\fn\subseteq \fgl(U\oplus U')$ is the subalgebra of all strictly block upper-triangular matrices.

\begin{remark}
    For classes in $\rH^{>0}(M,\cO_M)$, we do not know if the cup product induced from Lie algebra cohomology agrees with the cup product on group cohomology since the isomorphism between Koszul complexes that we gave in Proposition~\ref{prop:koszul-isom} is not multiplicative.
\end{remark}

As in Section~\ref{ss:lie-coh}, we denote $\alpha_{ij} = u_i\otimes v_j$ for $i = 1,\dots,m$ and $\theta_j = \alpha_{m+1,j} = u_{m+1}\otimes v_j$. 
For any $(m+1)$-multiset $N = \{1^{\nu_1},\dots,n^{\nu_n}\}$ of $\{1,\dots,n\}$, we define
\begin{align*}
    \tilde q_N &= \nu_1!\dots\nu_n! \sum_{\{i_1,\dots,i_{m},j\} = N} \alpha_{1i_1}\dots\alpha_{mi_m}\theta_{j}\in \rH^0(M,\mathcal O_M)_{1,1,\dots,1}\\
 q_N  &= \nu_1!\dots\nu_n!\cdot \tilde q_N.
\end{align*}
where the sum runs over all permutations of the multiset $N$ and \[\rH^0(M,\cO_M)_{1,1,\ldots,1}= \rH^0(M,\cO_M)\cap E_{1,1,\ldots,1}\] when realized as a subspace of $E$.

We recall the main theorem from the introduction that we have to prove in this section.

\begin{theorem}\label{thm: ring-generatorandrelations}

Let $M$ be a supertorus of dimension $m|n$. Then

\begin{itemize}
    \item[(i)] $\rH^0(M,\mathcal O_M)$ is generated as a $\C$-algebra by $\alpha_{ij}$ and $q_{N}$ where $N$ is an $(m+1)$-multiset of $\{1,\dots,n\}$.
    \item[(ii)] We have the following complete list of relations among the generators of $\rH^0(M,\mathcal{O}_M)$:
    \begin{itemize}
\item For a fixed $1\leq i\leq m$ and a $(m+2)$-multiset $N$ of $\{1, \dots,n\}$ the `linear' relations 
\[
\sum_{j \in N}\alpha_{ij}q_{N \setminus \{j\}}=0.
\]

\item For fixed multisets $N_1,N_2$ of size $m$ and $m+2$ respectively, the `quadratic' relations 
$$\sum_{j\in N_2}q_{N_1\cup \{j\}}q_{N_2\backslash \{j\}}=0.$$
\end{itemize}
\end{itemize}
\end{theorem}

\begin{example}
    Let $M$ be the supertorus of dimension $1|n$. Write $\alpha_i:=\alpha_{1i}$ for $1\le i\le n$. By Theorem~\ref{thm: ring-generatorandrelations}(i), the $\C$-algebra $\rH^0(M,\mathcal O_M)$ is generated by $\alpha_{i}$ and $q_{\{i,j\}} = \alpha_i\theta_j + \alpha_j\theta_i$. If $i=j$, we have
    \[
        q_{\{i^2\}} = 2\tilde q_{\{i^2\}} = 2\alpha_i\theta_i.
    \]    
    By Theorem~\ref{thm: ring-generatorandrelations}(ii), the generators $\alpha_i,q_N$ satisfy the following two types of relations:
    \begin{enumerate}
        \item Given $1 \le i \le n$ and a $3$-multiset $\{j,k, \ell\}$ of $\{1,\dots,n\}$ the linear relation is
        \[
        \alpha_{j}q_{\{k,\ell\}} + \alpha_{k} q_{\{i,\ell\}} + \alpha_{\ell} q_{\{j,k\}} = 0.
        \]  
        \item For $N_1 = \{l\}, N_2= \{i,j,k\}$, the quadratic relation is
            \[
              q_{il}q_{jk}+q_{jl}q_{ik}+q_{kl}q_{ij} = 0.
            \]
    \end{enumerate}

    This recovers \cite[Theorem 4.9]{RabinRhoadesKim2023}. Note that the generators in \cite{RabinRhoadesKim2023} are the reduced generators $\tilde q_N$ in our notation. 
\end{example}

The proof of Theorem~\ref{thm: ring-generatorandrelations} is now carried out in two parts in the following subsections.

\subsection{Generators of $\rH^0(M,\mathcal O_M)$}

We consider the full action of $\fgl(U \oplus U') \times \fgl(V)$ on $E$. We represent elements of this Lie algebra as pairs of matrices with respect to the chosen bases. As our Cartan subalgebra, we choose the subalgebra of pairs of diagonal matrices, and our Borel subalgebra is the subalgebra of pairs of upper-triangular matrices. We let $\fu$ denote the nilpotent subalgebra consisting of pairs of strictly upper-triangular matrices in these two factors. We will denote weights as sequences $(a;b)$ where $a \in \C^{m+1}$ and $b \in \C^n$.

As discussed earlier in the proof of Theorem~\ref{Irreducible-Decomposition}, the Cauchy identity gives a $\fgl(U\oplus U')\times \fgl(V)$-equivariant decomposition of $E$ into irreducible representations:
\[
  E \cong \bigoplus_\lambda \bS_\lambda(U \oplus U') \otimes \bS_{\lambda^T}(V).
\]

Since $\bS_\lambda(U \oplus U') \otimes \bS_{\lambda^T}(V)$ is irreducible, there is a 1-dimensional subspace of highest weight vectors of weight $(\lambda;\lambda^T)=(\lambda_1,\dots,\lambda_{m+1}; \lambda^T_1,\dots,\lambda^T_m)$ and it generates $\bS_\lambda(U \oplus U') \otimes \bS_{\lambda^T}(V)$.

Concretely, this means that if $v$ is a highest weight vector and $x \in \fgl(U \oplus U') \times \fgl(V)$ is the pair of diagonal matrices with diagonal entries $x_1,\dots,x_{m+1}$ in the first matrix and $y_1,\dots,y_n$ in the second matrix, then
\[
  x \cdot v = \left(\sum_{i=1}^{m+1} x_i \lambda_i + \sum_{j=1}^n y_j \lambda^T_j\right)v.
\]

We will denote $x_{ij}$ (resp. $y_{ij}$) as the matrix in $\fgl(U\oplus U')$ (resp. $\fgl(V)$) whose $(i,j)$-entry is $1$ and other entries are all $0$.

\begin{proof}[Proof of Theorem~\ref{thm: ring-generatorandrelations}(i):]
  For any $\lambda$, we define $a_\lambda\in E$ by
\[
  a_\lambda := \prod_{i=1}^{m+1} \prod_{j=1}^{\lambda_i} \alpha_{ij}
\]
  where the product is taken subject to lexicographic order (choosing a different order only introduces a sign and will not materially affect our discussion). 
  
  We claim that $a_\lambda$ is a highest weight vector of $E$ of weight $(\lambda;\lambda^T)$. In fact, if $\alpha_{jk}$ is a factor of $a_\lambda$, the action of $\alpha_{ik}\frac{\partial}{\partial \alpha_{jk}}$ on $a_\lambda$ replaces $\alpha_{jk}$ by $\alpha_{ik}$. Thus, if $i<j$, since $\lambda_1\ge \lambda_2 \ge \cdots$ is decreasing, that action will either be zero if $\alpha_{jk}$ is not in $a_\lambda$, or produce a repeated $\alpha_{ik}$ which kills $a_\lambda$ in any case. As a result, when $i<j$, we have that $X_{i,j} = \sum_k \alpha_{ik}\frac{\partial}{\partial \alpha_{jk}}$ annihilates $a_\lambda$. Similarly, when $i<j$, the action of $y_{i,j}$ on $a_{\lambda}$ (which acts  via the operator $Y_{i,j}=\sum_k \alpha_{ki}\frac{\partial}{\partial \alpha_{kj}}$) is just $0$, i.e., we have $y_{ij}\cdot a_\lambda = 0$. It follows that $a_\lambda$ is killed by $\fu$, the nilpotent subalgebra consisting of pairs of strictly upper-triangular matrices.  
  
  Next,  we have $x_{ii}\cdot a_\lambda = \lambda_ia_\lambda$ and $y_{jj}\cdot a_\lambda = \lambda^T_ja_\lambda$.
  It follows that if $(x,y)$ is the pair of diagonal matrices with diagonal entries $x_1,\dots,x_{m+1}$ and $y_1,\dots,y_n$, respectively, then
  \[
        x\cdot a_\lambda  = \left(\sum_{i=1}^{m+1} x_i \lambda_i + \sum_{j=1}^n y_j \lambda^T_j\right)a_\lambda.
  \]
  These calculations prove the claim.
  
  Next, from the definition, 
  \[
  \rH^0(\fn,E) = \{e\in E\ |\ xe=0 \text{ for all } x\in\fn\}
  \]
  is a subspace of $E$, which is a representation of the Lie subalgebra $\fgl(U)\times \fgl(V)$. Note that $\fn \subseteq \fu$, and $\fu \cap (\fgl(U) \times \fgl(V))$ is the subspace of pairs of upper-triangular matrices in this smaller subalgebra. In particular, if a representation appears in $\rH^0(\fn,E)$, any of its highest weight vectors is a scalar multiple of one of the $a_\lambda$.
  
  The span of $\alpha_{ij}$ and $q_N$ is a $\fgl(U) \times \fgl(V)$-subrepresentation of $\rH^0(\fn,E)$, and hence the subalgebra they generate is also a $\fgl(U) \times \fgl(V)$-subrepresentation. So it will suffice to show that each of the $a_\lambda$ can be generated by $\alpha_{ij}$ and $q_N$. We can do this explicitly:
  \[
    a_\lambda = \pm \left(\prod_{i=1}^{\lambda_{m+1}} q_{\{i,\dots,i\}}\right) \cdot \left(\prod_{i=1}^m \prod_{j=\lambda_{m+1}+1}^{\lambda_i} \alpha_{ij}\right),
  \]
  with the understanding that empty products give 1. This completes the proof.
\end{proof}

\subsection{Relations among the generators}\label{ss:rel_gen}

We now want to understand the relations between these generators.  In this section, we will show that the relations in Theorem~\ref{thm: ring-generatorandrelations}(ii) are in fact valid. The proof that they are a complete set of relations can be found in subsection~\ref{appendixA}.

In this section, given a multiset $N$, we use $\sum_{j \in N}$ to denote summation with multiplicities,  i.e., $j$ will be taken $\nu_j$ times in the sum where $\nu_j$ is its multiplicity in $N$, and $\sum^{\sim}$ to denote summation without multiplicities. 

\begin{example}
Let $N = \{1^2,2^3,3\} = \{1,1,2,2,2,3\}$. Then we have
\[
    \sum_{j\in N}q_{N \setminus \{j\}} = 2q_{12^33} + 3q_{1^22^23}+q_{1^22^3}, \quad \sum_{j\in N}^\sim q_{N \setminus \{j\}} = q_{12^33} + q_{1^22^23}+q_{1^22^3}. \qedhere
\]
\end{example}

\begin{lem}
    For any $(m+2)$-multiset $N$ of $\{1,\dots,n\}$ and $i = 1,\dots,m$, we have
    \[
        \sum_{j \in N} \alpha_{ij}q_{N \setminus \{j\}} = 0.
    \]
\end{lem}

\begin{proof}
    Without loss of generality, we can assume $i = 1$. Rewrite the sum as a linear combination of basis vectors:
    \[
        \sum_{j \in N} \alpha_{ij} q_{N \setminus \{j\}} = C\sum_{j \in  N}^\sim \alpha_{ij}\tilde q_{N \setminus \{j\}} = \sum_{1\leq k,l,j_2,\dots,j_{m+1}\leq n,\ k<l}^\sim c_{kl}^{j_2,\dots,j_{m+1}}\cdot \alpha_{1k}\alpha_{1l}\alpha_{2j_2}\dots \alpha_{mj_{m}}\theta_{j_{m+1}}.
    \]
    where  $C =  \nu_1!\dots\nu_n!$ if we write $N = \{1^{\nu_1},\dots,n^{\nu_n}\}$.
    
    For any basis vector $\alpha_{ik}\alpha_{il}\alpha_{ij_2}\dots \alpha_{ij_{m+1}}$ on the right, there are only two terms on the left containing it, namely $\alpha_{ik}q_{\{l,j_2,\dots,j_{m+1}\}}$ and $\alpha_{il}q_{\{k,j_2,\dots,j_{m+1}\}}$. Since they have opposite signs, we conclude that $c_{kl}^{j_2,\dots,j_{m+1}} = 0$.
\end{proof}

\begin{lem}\label{pluckerrel}
     Fix $2$ multisets $N_1, N_2$ of $\{1,\dots,n\}$ of sizes $m$ and $m+2$, respectively. Write $N_2 = \{1^{\nu_1},\dots,n^{\nu_n}\}$. Then
    \[
        \sum_{j \in N_2} q_{N_1 \cup \{j\}} q_{N_2 \setminus \{j\}} = 0.
    \]
\end{lem}

We first check it in the case when $N_1,N_2$ are ordinary sets:

\begin{lem}\label{ordpluker}
    Assume $n = 2m+2$. Fix $2$ disjoint (ordinary) sets $N_1, N_2$ of sizes $m$ and $m+2$, respectively. Then
    \[
        \sum_{j \in N_2} q_{N_1 \cup \{j\}} q_{N_2 \setminus \{j\}} = 0.
    \]
\end{lem}

\begin{proof}
    Rewrite it as linear combination of basis $\alpha_{1i_1}\alpha_{1j_1}\dots \alpha_{mi_m}\alpha_{mj_m}\alpha_{m+1,i_{m+1}}\alpha_{m+1,j_{m+1}}$ with $i_k < j_k$. Rewrite them in pairs: $(\alpha_{1i_1}\alpha_{1j_1})\cdots (\alpha_{mi_m}\alpha_{mj_m})(\alpha_{m+1,i_{m+1}}\alpha_{m+1,j_{m+1}})$. Then there are $m+1$ pairs in which exactly $m$ of them contain an $\alpha_{kl}$ such that $l\in N_1$. Denote the remaining pair by $(\alpha_{ki}\alpha_{kj})$. We know that only $q_{N_1 \cup \{j\}} q_{N_2 \setminus \{j\}}$ and $q_{N_1 \cup \{i\}} q_{N_2 \setminus \{i\}}$ contribute to the monomial and they have the opposite signs.
\end{proof}

We reduce the general case to the special case by considering the evaluation morphism.

\begin{proof}[Proof of Lemma \ref{pluckerrel}]
    Given an $m$-multisets $N_1 = \{1^{\nu_1},\dots,n^{\nu_n}\}$ and a $(m+2)$-multiset $N_2 = \{1^{\mu_1},\dots,n^{\mu_n}\}$, we consider an auxiliary vector space $W = \C w_1\oplus\dots\oplus \C w_{2m+2}$ and an assignment $\lambda\colon \{1,\dots,2m+2\}\to [n]$ which maps $1,2,3,\dots,m$ to $1^{\nu_1}|\dots|n^{\nu_n}$ and $m+1,m+2,\dots,2m+2$ to $1^{\mu_1}|\dots|n^{\mu_n}$.

    Then the assignment induces an evaluation ring map:
    \[
        \phi\colon \bigwedge^\bullet ((U\oplus U')\otimes W) \to \bigwedge^\bullet ((U\oplus U')\otimes V).
    \]
    We claim that for any ordinary set $\{i_1,\dots,i_{m+1}\}$, we have $\phi(q_{i_1,\dots,i_{m+1}}) = q_{\{\lambda(i_1),\dots,\lambda(i_{m+1})\}}$. In fact, let $\alpha_{1j_1}\dots \alpha_{m+1,j_{m+1}}$ be a monomial in $q_{\{\lambda(i_1),\dots,\lambda(i_{m+1})\}}$. Then a monomial $\alpha_{1\sigma(i_1)}\dots \alpha_{m+1, \sigma(i_{m+1})}$ in $q_{i_1,\dots,i_{m+1}}$ contributes to $\alpha_{1j_1}\dots \alpha_{m+1,j_{m+1}}$ if and only if $\sigma\in S_{\{i_1,\dots,i_{m+1}\}}$ is an internal permutation over the multiset $\{\lambda(i_1),\dots,\lambda(i_{m+1})\}$ (in which case the coefficient is $1$). If we write $\{\lambda(i_1),\dots,\lambda(i_{m+1})\} = \{1^{\tau_1},\dots,n^{\tau_n}\}$,  then there are  exactly $\tau_1!\dots\tau_n!$ such internal permutations, which proves the claim.

    By applying Lemma \ref{ordpluker} to $M_1 = \{1,\dots,m\}$ and $M_2 = \{m+1,\dots,2m+2\}$, we have 
    \[
        \sum_{j \in M_2} q_{M_1 \cup \{j\}} q_{M_2 \setminus \{j\}} = 0.
    \]
    The statement follows by applying $\phi$ to both sides.
\end{proof}

\subsection{Completeness of the relations}
\label{appendixA}

The generators $\alpha_{ij} \in \rH^0(\fn, E)$ span the representation $U \otimes V$ while the generators $q_N$ span the representation $\bigwedge^m U \otimes U' \otimes S^{m+1}(V)$. We will ignore $U'$ since it is 1-dimensional (and the powers of $U'$ we need to use are easily recoverable).

Hence, when $m$ is even, we have a $\fgl(U) \times \fgl(V)$-equivariant surjective ring homomorphism
\[
  \bigwedge^\bullet(U \otimes V) \otimes \bigwedge^\bullet(\bigwedge^m U \otimes S^{m+1} V) \xrightarrow{\Phi} \rH^0(\fn, E),
\]
and when $m$ is odd, we have instead
\[
  \bigwedge^\bullet(U \otimes V) \otimes S^\bullet(\bigwedge^m U \otimes S^{m+1} V) \xrightarrow{\Phi} \rH^0(\fn, E).
\]

Our goal is to compute generators for the ideal $\ker \Phi$.

Let's fix $U$ and consider $V$ as a variable vector space (see \cite[\S 5.3]{expos} for this perspective). Then both the domain and the codomain of $\Phi$ are polynomial functors in $V$, and so we can apply the transpose operation $\Omega$ (see \cite[\S 7.4]{expos}, where $\Omega (-)$ is written as $(-)^\dagger$). This is an autoequivalence of the category of polynomial functors such that $\Omega\bS_\lambda=\bS_{\lambda^T}$ (in particular, swaps exterior and symmetric powers). Given homogeneous functors $F$ and $G$, we have (see \cite[(7.4.8)]{expos})
\[
  \Omega(F \circ G) = 
  \begin{cases} 
  F \circ \Omega G & \text{if $\deg(G)$ is even} \\ 
  \Omega F \circ \Omega G & \text{if $\deg(G)$ is odd} 
  \end{cases}.
\]
Let $\Psi = \Omega(\Phi)$. Then we have (independent of the parity of $m$) a surjective $\fgl(U) \times \fgl(V)$-equivariant map:
\[
    S^\bullet(U \otimes V) \otimes S^\bullet(\bigwedge^m U \otimes \bigwedge^{m+1} V) \xrightarrow{\Psi} \Omega (\rH^0(\fn, E)).
  \]
  We have
  \[
    \Omega(\rH^0(\fn, E)) = \bigoplus_{\ell(\lambda) \le m+1} \bS_{(\lambda_1,\dots,\lambda_m)}(U) \otimes \bS_{\lambda}(V).
  \]
  This is a commutative ring which we can construct geometrically. 

  Let $X$ denote the Grassmannian of rank $m+1$ quotients of $V$. It has a tautological exact sequence of vector bundles
  \[
    0 \to \cR \to V \otimes \cO_X \to \cQ \to 0
  \]
  where $\rank \cR = \dim V - m - 1$ and $\rank \cQ = m+1$. We will use the setup of \cite[Chapter 6]{weyman}. Define
  \[
    \eta_1 = U \otimes \cQ, \qquad \eta_2 = \bigwedge^m U \otimes \bigwedge^{m+1} \cQ, \qquad \eta = \eta_1 \oplus \eta_2.
  \]
  Note that $\eta_2$ is a line bundle; in fact, it is the ample generator which gives $X$ its Pl\"ucker embedding. We can realize $\eta_1,\eta_2$ as quotients of trivial vector bundles. We record these as short exact sequences
  \begin{align*}
    0\to \xi_1 \to (U \otimes V) \otimes \cO_X \to \eta_1 \to 0,\qquad
    0\to \xi_2 \to (\bigwedge^m U \otimes \bigwedge^{m+1} V) \otimes \cO_X \to \eta_2 \to 0.
  \end{align*}
  Note that $\xi_1 = U \otimes \cR$ (but $\xi_2$ does not have a simple description). Define
  \[
    \xi = \xi_1 \oplus \xi_2.
  \]

  Before proceeding, we will need the Borel--Weil--Bott theorem for $X$ (see \cite[\S 4.1]{weyman}) in the following form.

  \begin{theorem}[Borel--Weil--Bott]
Given the bundle $\cE = \bS_\lambda \cQ \otimes \bS_\mu \cR$ we consider the sequence
  \[
    \alpha = (\lambda_1,\dots,\lambda_{m+1},\mu_1,\mu_2,\dots).
  \]
  Let $\rho = (0,-1,-2,\dots)$.
  \begin{enumerate}
  \item If $\alpha + \rho$ has repetitions, then $\rH^d(X, \cE)=0$ for all $d$.
  \item Otherwise, there is a unique permutation $\sigma$ such that $\beta := \sigma(\alpha+\rho)-\rho$ is weakly decreasing. Then
    \[
      \rH^{\ell(\sigma)}(X, \cE) \cong \bS_\beta V
    \]
    and all other cohomology groups vanish, where $\ell(\sigma) = |\{i<j \mid \sigma(i)>\sigma(j)\}|$.
  \end{enumerate}
\end{theorem}

  \begin{proposition}
    We have an isomorphism of algebras
  \[
    \rH^0(X, S^\bullet(\eta)) = \Omega(\rH^0(\fn, E))
  \]
  and, for all $i>0$, we have
  \[
    \rH^i(X, S^\bullet(\eta))=0.
  \]
  In particular, $\Omega(\rH^0(\fn, E))$ is a normal ring with rational singularities.
  \end{proposition}

  \begin{proof}
    Both algebras are isomorphic as $\fgl(U) \times \fgl(V)$-representations and are multiplicity-free.    We know what the right side looks like. For the left side, we have
    \begin{align*}
      \rH^0(X, S^\bullet(\eta))
      &= \bigoplus_{\mu, d} \rH^0(X, \bS_\mu U \otimes \bS_\mu \cQ \otimes (\bigwedge^m U)^{\otimes d} \otimes (\bigwedge^{m+1} \cQ)^{\otimes d})\\
      &= \bigoplus_{\mu, d} \rH^0(X, \bS_{d+\mu_1,\dots,d+\mu_m} U \otimes \bS_{d+\mu_1,\dots,d+\mu_m, d} \cQ)\\
      &= \bigoplus_{\ell(\lambda)\le m+1} \rH^0(X, \bS_{\lambda_1,\dots,\lambda_m} U \otimes \bS_\lambda \cQ)\\
      &= \bigoplus_{\ell(\lambda)\le m+1} \bS_{\lambda_1,\dots,\lambda_m} U \otimes \bS_\lambda V.
    \end{align*}
    where in the penultimate equality, we use the reparametrization $\lambda = (d+\mu_1,\dots,d+\mu_m,d)$.

    So we just need to show that the left hand side is generated by $\rH^0(X, \eta_1) = U \otimes V$ and $\rH^0(X, \eta_2) = \bigwedge^m U \otimes \bigwedge^{m+1} V$. By construction, $\rH^0(X, S^\bullet(\eta))$ is an integral domain.

    To this end, $\rH^0(X, S^\bullet(\eta_i))$ is generated by $\rH^0(X, \eta_i)$ for $i=1,2$. For $i=1$, this is normality of determinantal rings (see, for instance, \cite[Theorem 6.1.4]{weyman}); for $i=2$, this is projective normality of the Grassmannian under the Pl\"ucker embedding, (see, for instance, \cite[Theorem 7.1.2(a)]{weyman} with $\lambda = (1,\dots,1)$), but the next argument with highest weight vectors also works.

    In general, consider the map
    \[
      \rH^0(X, \bS_{\lambda} U \otimes \bS_\lambda \cQ) \otimes \rH^0(X, S^j(\bigwedge^m U \otimes \bigwedge^{m+1} \cQ)) \to \rH^0(X, \bS_{\lambda} U \otimes \bS_\lambda \cQ \otimes S^j(\bigwedge^m U \otimes \bigwedge^{m+1} \cQ)).
    \]
    The target is an irreducible representation and its highest weight is the sum of the highest weights for the two factors in the tensor product in the domain (this is an instance of a Cartan product). In particular, if we multiply the sections given by the highest weight vectors of the two representations in the domain, then we get a nonzero highest weight vector for the target. This proves the generation statement we wanted.
    
  The vanishing statement immediately follows from Borel--Weil--Bott. For the last statement, we consider the projection map from the total space of $\eta$ to the affine space $((U \otimes V) \oplus (\bigwedge^m U \otimes \bigwedge^{m+1} V))^*$. This map is birational onto its image, and it follows from \cite[Theorem 5.1.3]{weyman} that $\rH^0(X, S^\bullet \eta)$ is the normalization of its support and that it has rational singularities. (In fact, we just showed that $\rH^0(X, S^\bullet \eta)$ is generated by $\rH^0(X, \eta)$ so it coincides with its support.)
  \end{proof}
In particular, the Tor groups of $\rH^0(X, S^\bullet \eta)$ over $S^\bullet(U \otimes V) \otimes S^\bullet(\bigwedge^m U\otimes \bigwedge^{m+1} V)$ can be computed as in \cite[Theorem 5.1.2]{weyman}. We are interested in $\Tor_1$, which is the space of minimal generators for $\ker \Psi$, and this is given by
  \[
    \bigoplus_j \rH^{j}(X, \bigwedge^{j+1} \xi).
  \]

  First, we narrow which terms give nonzero contribution.

  \begin{proposition} \label{prop:nonzero-contribution}
    We have
    \[
      \bigoplus_j \rH^{j}(X, \bigwedge^{j+1} \xi) = \rH^1(X, \xi_1 \otimes \xi_2) \oplus \rH^1(X, \bigwedge^2 \xi_2).
    \]
  \end{proposition}

  To prove this, we first use the Cauchy identity \cite[Theorem 2.3.2]{weyman}
  \[
    \bigwedge^\bullet \xi_1 = \bigoplus_\lambda \bS_{\lambda^T} U \otimes \bS_\lambda \cR.
  \]
  The sum is over $\lambda$ such that $\lambda_1 \le m$ and $\ell(\lambda) \le \dim V - m-1$.
  Take the $i$th exterior power of the short exact sequence defining $\xi_2$ to get
  \begin{align} \label{eqn:ext-xi2}
    0 \to \bigwedge^i \xi_2 \to \cF(i)_0 \to \cdots \to \cF(i)_i \to 0
  \end{align}
  where
  \[
    \cF(i)_j = \bigwedge^{i-j}(\bigwedge^{m+1} V) \otimes  S^j(\bigwedge^{m+1} \cQ).
  \]
  We will use this to consider the cohomology groups $\rH^{|\lambda|+i-1}(X, \bS_\lambda \cR \otimes \bigwedge^i \xi_2)$ together with the following standard argument.

  \begin{lemma} \label{lem:LES-vanish}
    Let $0 \to \cE \to \cF_0 \to \cdots \to \cF_n \to 0$ be a long exact sequence of coherent sheaves. For fixed $d$, if $\rH^{d-j}(\cF_j)=0$ for all $j$, then $\rH^d(\cE)=0$.
  \end{lemma}

  \begin{proof}
    Do induction on $n$; to use the induction step, we replace the last two terms $\cF_{n-1} \to \cF_n$ by the kernel of the map.
  \end{proof}

  \begin{lemma} \label{lem:F-terms-vanish}
    If $\lambda_1 \le m$, then $\rH^d(X, \bS_\lambda \cR \otimes S^j(\bigwedge^{m+1} \cQ))=0$ for all $d>0$.
  \end{lemma}

  \begin{proof}
    The weight associated with the homogeneous bundle in question is
    \[
      ( \underbrace{j, \dots, j}_{m+1}, \lambda_1, \lambda_2, \dots).
    \]
    Adding $\rho$ gives
    \[
      (j, j-1, \dots, j-m, \lambda_1 -m-1, \lambda_2 - m-2,\dots).
    \]
    Since $j \ge 0$ and $\lambda_1 \le m$, this sequence either has a repetition or is decreasing, so the sheaf has no higher cohomology. 
  \end{proof}
  
  \begin{lemma} \label{lem:F-surject}
    For all $i \ge 2$, the map $\rH^0(X, \cR \otimes \cF(i)_{i-1}) \to \rH^0(X, \cR \otimes \cF(i)_i)$ is surjective.
  \end{lemma}

  \begin{proof}
    The map $\cF(i)_{i-1} \to \cF(i)_i$ is the composition 
    \[
      \bigwedge^{m+1} V \otimes  S^{i-1}(\bigwedge^{m+1} \cQ) \to \bigwedge^{m+1} \cQ \otimes  S^{i-1}(\bigwedge^{m+1} \cQ) \to S^i(\bigwedge^{m+1} \cQ),
      \]
    where the first map is the natural surjection and the second map is multiplication.     We have
    \begin{align*}
      \rH^0(X, \cR \otimes \cF(i)_{i-1}) &= \bigwedge^{m+1} V \otimes \bS_{(i-1,\dots,i-1,1)} V, \\
      \rH^0(X, \cR \otimes \cF(i)_i) &= \bS_{(i,\dots,i, 1)} V.
    \end{align*}
    The product of the highest weight vectors from $\bigwedge^{m+1} V$ and $\bS_{(i-1,\dots,i-1,1)} V$ gives a (nonzero) highest weight vector for $\bS_{(i,\dots,i,1)} V$; since the latter is irreducible, this proves the surjectivity statement.
  \end{proof}

Now we can finish the proof of Proposition~\ref{prop:nonzero-contribution}.
  
  \begin{proof}[Proof of Proposition~\ref{prop:nonzero-contribution}]
    Consider the cohomology
    \[
      \rH^{d+e-1}(X, \bigwedge^d \xi_1 \otimes \bigwedge^e \xi_2) = \bigoplus_{|\lambda|=d} \bS_{\lambda^T} U  \otimes \rH^{d+e-1}(X, \bS_\lambda \cR \otimes \bigwedge^e \xi_2).
    \]
    First suppose $d \ge 2$. In this case, we take \eqref{eqn:ext-xi2} with $i=e$ and tensor with $\bS_\lambda \cR$. By Lemma~\ref{lem:F-terms-vanish}, the terms $\bS_\lambda \otimes \cF(e)_j$ all have vanishing higher cohomology. Hence Lemma~\ref{lem:LES-vanish} implies that $\rH^{d+e-1}(X, \bS_\lambda \cR \otimes \bigwedge^e \xi_2)=0$ since $d+e-1-j > 0$ for $j \le e$.

    Now consider $d=1$ and $e \ge 2$. Let $\cK(e)$ be the kernel of $\cF(e)_{e-1} \to \cF(e)_e$.  Then we consider the long exact sequence
    \[
      0 \to \cR \otimes \bigwedge^e \xi_2 \to \cR \otimes \cF(e)_0 \to \cdots \to \cR \otimes \cF(e)_{e-2} \to \cR \otimes \cK(e) \to 0.
    \]
    For $j \le e-2$, we can again apply Lemma~\ref{lem:F-terms-vanish} to see that $\rH^{e-j}(X, \cR \otimes \cF(e)_j)=0$. By Lemma~\ref{lem:F-surject}, $\rH^1(X, \cR \otimes \cK(e))=0$. Hence Lemma~\ref{lem:LES-vanish} implies that $\rH^{e}(X, \cR \otimes \bigwedge^e \xi_2)=0$.

    Finally, we consider $d=0$. The algebra $\rH^0(X, S^\bullet(\bigwedge^m U \otimes \bigwedge^{m+1} \cQ))$ is the homogeneous coordinate ring of the Pl\"ucker embedding of $X$. Hence, $\bigoplus_j \rH^{j-1}(X, \bigwedge^j \xi_2)$ computes the minimal generators of this ideal. But it is known to be generated by quadratic equations (the Pl\"ucker relations) so this is only nonzero for $j=2$.
  \end{proof}

  Let's make Proposition~\ref{prop:nonzero-contribution} more concrete. First, we can consider $\rH^0(X, S^\bullet \eta)$ as a bi-graded ring with $\deg \rH^0(X, S^i \eta_1 \otimes S^j \eta_2) = (i,j)$. Hence we can explicitly identify the space of degree $(1,1)$ relations as
  \begin{align*}
    \rH^1(X, \xi_1 \otimes \xi_2) &= \ker  \left( \rH^0(X, \eta_1) \otimes \rH^0(X, \eta_2) \to \rH^0(X, \eta_1 \otimes \eta_2) \right)\\
                                  &=\ker \left( (U \otimes V ) \otimes (\bigwedge^m U\otimes  \bigwedge^{m+1} V) \to \bS_{2,1,\dots,1} U \otimes \bS_{2,1,\dots,1} V \right).
  \end{align*}
  Since $\dim U = m$, we have $U \otimes \bigwedge^m U = \bS_{2,1,\dots,1} U$ and by the Pieri rule (see \cite[(3.10)]{expos}), we have
  \[
    V \otimes \bigwedge^{m+1} V \cong \bigwedge^{m+2} V \oplus \bS_{2,1,\dots,1} V.
  \]
  Since the map is surjective, we conclude that
  \[
    \rH^1(X, \xi_1 \otimes \xi_2) = \bigwedge^m U \otimes U \otimes \bigwedge^{m+2} V.
  \]
  If we apply $\Omega$, this gives the subspace
  \[
    \bigwedge^m U \otimes U \otimes S^{m+2} V \subseteq \ker \Phi.
  \]

  This is spanned by the ``linear'' relations (fix $1 \le i \le m$ and an $(m+2)$-multiset $N$ of $\{1,\dots,n\}$):
  \[
    \sum_{j \in N} \alpha_{ij} \otimes q_{N \setminus \{j\}}.
  \]
The highest weight vector is given by $i=1$ and $N = \{1,\dots,1\}$.

Next, as stated in the proof, $\rH^1(X, \bigwedge^2 \xi_2)$ are the Pl\"ucker relations on $\bigwedge^m U \otimes \bigwedge^{m+1} V$. If we apply $\Omega$, then these equations are the kernel of (if $m$ is even)
\[
  \bigwedge^2(S^{m+1} V) \to \bS_{(m+1,m+1)} V
\]
or (if $m$ is odd)
\[
  S^2(S^{m+1} V) \to \bS_{(m+1,m+1)} V.
\]
In both cases, call the map $F$.

These equations can be written as follows: fix two multisets $N_1$ and $N_2$ of sizes $m$ and $m+2$, respectively. Then the odd analogues of the ``Pl\"ucker relations'' are
\[
  \sum_{j \in N_2} q_{N_1 \cup \{j\}} q_{N_2 \setminus \{j\}}.
\]

\begin{proposition} \label{prop:odd-plucker}
  The odd Pl\"ucker relations span $\ker F$.
\end{proposition}

To prove this, we consider a variation first. Pick $a, b$ and consider the map
\[
\sigma_{a,b} \colon  S^{a+1} V \otimes S^{b-1} V \to S^a V \otimes S^b V
\]
which is defined as the composition
\begin{align*}
  S^{a+1} V \otimes S^{b-1} V &\xrightarrow{\Delta \otimes 1} S^a V \otimes V \otimes S^{b-1} V \\
  &\xrightarrow{1 \otimes m} S^a V \otimes S^b V
\end{align*}

Here $m \colon V \otimes S^{b-1} V \to S^b V$ is the multiplication map on polynomials and $\Delta \colon S^{a+1} V \to S^a V \otimes V$ is the comultiplication map defined by
\[
  \Delta(x_{i_1} \cdots x_{i_{a+1}}) = \sum_{j=1}^{a+1} x_{i_1} \cdots \widehat{x_{i_j}} \cdots x_{i_{a+1}} \otimes x_{i_j}.
\]
These are $\GL(V)$-linear maps. The following is likely well-known but we give a proof for completeness.

\begin{proposition}
If $a \ge b$, then  $\sigma_{a,b}$ is injective.
\end{proposition}

\begin{proof}
  Since all Schur functors $\bS_\lambda$ in the domain satisfy $\lambda_3=0$ by Pieri's rule, it suffices to check this when $V = \C^2$. Let $\{x,y\}$ be a basis for $\C^2$ and work with the Borel subgroup of $\GL_2$ which is upper-triangular with respect to this basis. Then it also suffices to check that $\ker \sigma_{a,b}$ contains no highest weight vectors. Every $\bS_\lambda$ appearing satisfies $\lambda_1 \ge a+1$, so we only need to consider the weight spaces $(\mu_1,\mu_2)$ with $\mu_1 \ge a+1$. This has a basis consisting of the vectors $x^i y^{a+1-i} \otimes x^j y^{b-1-j}$ where $i+j = \mu_1$. In particular, $i \ge \mu_1 - j \ge a-b+2 > 0$. We have
  \[
    \sigma_{a,b}(x^i y^{a+1-i} \otimes x^j y^{b-1-j}) = i x^{i-1} y^{a+1-i} \otimes x^{j+1} y^{b-1-j} + (a+1-i) x^i y^{a-i} \otimes x^j y^{b-j}.
  \]
  Since the first term has nonzero coefficient, we see that $\sigma$ is triangular with respect to this basis with nonzero entries on the diagonal, and hence injective.
\end{proof}

\begin{proof}[Proof of Proposition~\ref{prop:odd-plucker}]
  For a multiset $N = \{n_1,\dots,n_{m+1}\}$, we identify $q_N$ with the monomial $x_{n_1} \cdots x_{n_{m+1}}$. The span of the odd Pl\"ucker relations are the image of $\sigma_{m+1,m+1}$ composed with the projection of $S^{m+1} V \otimes S^{m+1} V$ onto $\bigwedge^2 (S^{m+1} V)$ (when $m$ is even) or $S^2(S^{m+1} V)$ (when $m$ is odd).   Using the Pieri rule and the fact that $\sigma_{m+1,m+1}$ is injective, we see that its cokernel must be the Schur functor $\bS_{(m+1,m+1)}(V)$. Hence the cokernel of $F$ is also isomorphic to $\bS_{(m+1,m+1)}(V)$.
\end{proof}

We can deduce from this the linear dependencies amongst the odd Pl\"ucker relations: the relations come from the intersection of their span with $S^2(S^{m+1} V)$ (when $m$ is even) or $\bigwedge^2(S^{m+1} V)$ (when $m$ is odd).

\section{Picard group of supertori} \label{Section:Picard}
As noted in \cite[Section 2.15]{BergveltRabin1999} and \cite[Section 3]{Rabin1995}, the even Picard group $\operatorname{Pic}_{\text{\rm ev}}(M)=\rH^1(M, \cO_{\text{\rm ev}}^{*})$ is the group of isomorphism classes of line bundles on $M$ under tensor product.
In this section, we compute the degree-zero even Picard group of the supertorus $\operatorname{Pic}^0_{\text{\rm ev}}(M)$ explicitly in terms of cocycles and coboundaries, as an application of the results in Section~\ref{ss:group-coh}. 

We first fix some notations. Let $\Lambda$ be the Grassmann algebra defined in Section~\ref{ss:supertori} and $E$ be the exterior algebra defined in Section~\ref{ss:group-coh}. 
Let $E_{{\rm ev}}$ (resp. $\Lambda_{\rm ev}$) denote the even part of $E$ (resp. $\Lambda$). Then $N_{\rm ev}=\Gamma(M, \cO_{\text{\rm ev}})=\cO(\C^m)\otimes_{\C} E_{{\rm ev}}$. Note that, as $\Lambda_{\rm ev}$-modules, we have
\[
E_{\rm ev} \cong \bigoplus_{\substack{J\subseteq \{1,\ldots,n\}\\ |J|\; {\rm even}}}\Lambda_{\rm ev} \oplus \bigoplus_{\substack{J\subseteq \{1,\ldots,n\}\\ |J|\; {\rm odd}}}\Lambda_{\rm odd} 
\]
via the map
\[
F = \sum_{\substack{0 \le k \le n\\k\; \text{even}\\ i_1 < \cdots < i_k}} f_{i_1,\ldots,i_k}\theta_{i_1}\theta_{i_2}\ldots\theta_{i_k} +
\sum_{\substack{1 \le k \le n\\ k\; \text{odd} \\ i_1 < \cdots < i_k}}\varphi_{i_1,\ldots,i_k}\theta_{i_1}\theta_{i_2}\ldots\theta_{i_k} \longmapsto ((f_{i_1,\ldots, i_k})_{k\; {\rm even}}, (\varphi_{i_1,\ldots,i_k})_{k \; {\rm odd}}) ,
\]
We then let
\[
\Psi:= \bigoplus_{i=1}^m \left(\bigoplus_{\substack{J\subseteq \{1,\ldots,n\}\\ |J|\; {\rm even}}}\Lambda_{\rm ev} \oplus \bigoplus_{\substack{J\subseteq \{1,\ldots,n\}\\ |J|\; {\rm odd}}}\Lambda_{\rm odd}\right) 
\]
as an abelian group. We will use it in this section to represent $m$ elements of $E_{\rm ev}$.

\begin{proposition}\label{prop:picard}
There  are isomorphisms of groups
\[\operatorname{Pic}^0_{\rm ev}(M)
\cong
\frac{\rH^1(G, \cO_{\rm ev})}{\rH^1(G,\mathbb{Z})} \\[6pt]
\cong
\left( \frac{ \langle
(f^{(i)}_J,\,
\phi^{(i)}_J)\in \Psi
:\;
J\subseteq \{1,\ldots,n\},\;
1\le i\le m,\;
C
\rangle
}{
\langle R\rangle
}
\right)
\Big/
\mathcal L ,
\]
where $C$ and $R$ are explicit constraints (\ref{eq:constraints}) and relations (\ref{eq:rel1}, \ref{eq:rel2}, \ref{eq:rel3}, \ref{eq:rel4}), respectively, and $\cL$ is an explicit lattice (\ref{eq:lattice}). 
\end{proposition}

\begin{proof}
By the same arguments as in Section \ref{ss:group-coh} (replacing $E$ by $E_{\rm ev}$ and $N$ by $N_{\rm ev}$), we get
\[
  \rH^1(M, \cO_{\text{\rm ev}})\cong \rH^1(G, N_{\rm ev}) \cong \rH^1(\mathcal{T}, \rH^0(\mathcal{S}, N_{\rm ev}))\cong \rH^1(\mathcal{T}, E_{\rm ev}). 
\]
By definition of group cohomology, we have
\[
  \rH^1(\mathcal{T},E_{\rm ev}) = \frac{\{f \colon \cT\to E_{\rm ev}\; :\; f(gt) = f(g) + gf(t) \ \text{for any} \ g,t \in \mathcal{T}\}}{\{f \colon \cT \to E_{\rm ev}:\text{there exists $e\in E_{\rm ev}$ such that $f(g) = e-ge$ for any $g \in \mathcal{T}$}\}}.  
\]
Since $\mathcal{T}=\langle T_1,\ldots,T_m\rangle$ is abelian, the numerator can be identified with
\[
    Z=\{(F_1,\ldots, F_m): F_i=f(T_i)\in E_{\rm ev}\; \text{such that}\; F_i + T_iF_j = F_j + T_jF_i\; \;\text{for all} \; i\neq j\}.
\]
We now write each $F_i\in E_{\rm ev}$ as
\[
F_i = \sum_{\substack{0 \le k \le n\\k\; \text{even}\\ i_1 < \cdots < i_k}} f^{(i)}_{i_1,\ldots,i_k}\theta_{i_1}\theta_{i_2}\ldots\theta_{i_k} +
\sum_{\substack{1 \le k \le n\\ k\; \text{odd} \\ i_1 < \cdots < i_k}}\phi^{(i)}_{i_1,\ldots,i_k}\theta_{i_1}\theta_{i_2}\ldots\theta_{i_k},
\]
where the coefficients $f^{(i)}_{I}$ are even elements of $\Lambda$ and the coefficients $\phi^{(i)}_{I}$ are odd elements of $\Lambda$.
The cocycle constraint
$(T_i-1)F_j=(T_j-1)F_i$ for $i\neq j$ yields

\begin{align*}
&\sum_{\substack{i_1<\cdots<i_k \\ k\ \text{even}}}
f^{(j)}_{i_1,\ldots,i_k}\,
\Delta^{(i)}_{i_1,\ldots,i_k}
+
\sum_{\substack{i_1<\cdots<i_k \\ k\ \text{odd}}}
\phi^{(j)}_{i_1,\ldots,i_k}\,
\Delta^{(i)}_{i_1,\ldots,i_k}
\\
&=
\sum_{\substack{i_1<\cdots<i_k \\ k\ \text{even}}}
f^{(i)}_{i_1,\ldots,i_k}\,
\Delta^{(j)}_{i_1,\ldots,i_k}
+
\sum_{\substack{i_1<\cdots<i_k \\ k\ \text{odd}}}
\phi^{(i)}_{i_1,\ldots,i_k}\,
\Delta^{(j)}_{i_1,\ldots,i_k},
\end{align*}
where 
\[
\Delta^{(i)}_{i_1,\ldots,i_k}
:=
(\theta_{i_1}+\alpha_{ii_1})
\cdots
(\theta_{i_k}+\alpha_{ii_k})
-
\theta_{i_1}\cdots\theta_{i_k}.
\]

We now compare coefficients of $\theta$ on both sides. Let $J\subsetneq \{1,\ldots,n\}$ be a subset and $\theta_J$ denote the corresponding theta product in order of the indices. For any $I=\{i_1,\ldots,i_k\}\subseteq\{1,\ldots,n\}$ with $I\supsetneq J$, we define \[ (-1)^{\varepsilon(J,I)}:=  (-1)^{j_1-1}\cdots(-1)^{j_\ell-1}, \] where $I\setminus J  = \{j_1,\dots,j_{\ell}\}$.
Then, comparing coefficients of $\theta_J$, we get 
\begin{equation}\label{eq:constraints}
\begin{aligned}
\sum_{\substack{I=\{i_1,\ldots,i_k\}\supsetneq J\\k \; \text{even}}}f^{(j)}_{i_1,\ldots, i_k}(-1)^{\varepsilon(J,I)}\prod_{i_l\notin J}\alpha_{ii_l}+\sum_{\substack{I=\{i_1,\ldots,i_k\}\supsetneq J\\k \; \text{odd}}}\phi^{(j)}_{i_1,\ldots, i_k}(-1)^{\varepsilon(J,I)}\prod_{i_l\notin J}\alpha_{ii_l}&\\=
\sum_{\substack{I=\{i_1,\ldots,i_k\}\supsetneq J\\k \; \text{even}}}f^{(i)}_{i_1,\ldots, i_k}(-1)^{\varepsilon(J,I)}\prod_{i_l\notin J}\alpha_{ji_l}+\sum_{\substack{I=\{i_1,\ldots,i_k\}\supsetneq J\\k \; \text{odd}}}\phi^{(i)}_{i_1,\ldots, i_k}(-1)^{\varepsilon(J,I)}\prod_{i_l\notin J}\alpha_{ji_l}.
\end{aligned}
\end{equation}
We denote the above collection of constraints by $C$.

A cocycle $(F_1,\ldots, F_m)$ is a coboundary if there exists $ \tilde{F}\in E_{\rm ev}$ such that $F_i = \tilde{F}-T_i\tilde{F}$. Expanding this out, we get

\begin{align*}
&\sum_{\substack{i_1<\cdots<i_k \\ k\ \mathrm{even}}}
f^{(i)}_{i_1,\ldots,i_k}\,\theta_{i_1}\cdots\theta_{i_k}
+
\sum_{\substack{i_1<\cdots<i_k \\ k\ \mathrm{odd}}}
\phi^{(i)}_{i_1,\ldots,i_k}\,\theta_{i_1}\cdots\theta_{i_k}
\\
&=
-\sum_{\substack{i_1<\cdots<i_k \\ k\ \mathrm{even}}}
\tilde{f}_{i_1,\ldots,i_k}\,\Delta^{(i)}_{i_1,\ldots,i_k}
-\sum_{\substack{i_1<\cdots<i_k \\ k\ \mathrm{odd}}}
\tilde{\phi}_{i_1,\ldots,i_k}\,\Delta^{(i)}_{i_1,\ldots,i_k} \; \;\text{for all}\; 1\leq i\leq m.
\end{align*}

As before, for $J\subsetneq \{1,\ldots,n\}$, by comparing coefficients of $\theta_J$, we obtain the following relations. 

If $J$ has an even number of elements, we have
\begin{subequations}
\begin{equation}\label{eq:rel1}
\begin{aligned}
     f^{(i)}_J = \sum_{\substack{\{i_1,\ldots,i_k\}\supsetneq J\\ k\; \text{even}}}\tilde{f}_{i_1,\ldots,i_k}(-1)^{\varepsilon(J,I)+1}\prod_{i_l\notin J}\alpha_{ii_l} + \sum_{\substack{\{i_1,\ldots,i_k\}\supsetneq J\\ k\; \text{odd}}}\tilde{\phi}_{i_1,\ldots,i_k}(-1)^{\varepsilon(J,I)+1}\prod_{i_l\notin J}\alpha_{ii_l}
\end{aligned}
\end{equation}
for all $1\leq i\leq m$.

If $J$ has an odd number of elements, we have
\begin{equation}\label{eq:rel2}
   \phi^{(i)}_J = \sum_{\substack{\{i_1,\ldots,i_k\}\supsetneq J\\ k\; \text{even}}}\tilde{f}_{i_1,\ldots,i_k}(-1)^{\varepsilon(J,I)+1}\prod_{i_l\notin J}\alpha_{ii_l} + \sum_{\substack{\{i_1,\ldots,i_k\}\supsetneq J\\ k\; \text{odd}}}\tilde{\phi}_{i_1,\ldots,i_k}(-1)^{\varepsilon(J,I)+1}\prod_{i_l\notin J}\alpha_{ii_l}    
\end{equation}
for all $1\le i\le m$.
Comparing the coefficient of $\theta_1\theta_2\cdots\theta_n$ yields
\begin{equation}\label{eq:rel3}
f^{(i)}_{\{1,\ldots,n\}}=0 \; \;\text{for all}\; 1\leq i\leq m   \quad \text{if}\; n\; \text{is even};
\end{equation}
\begin{equation}\label{eq:rel4}
\phi^{(i)}_{\{1,\ldots,n\}}=0 \; \;\text{for all}\; 1\leq i\leq m \quad \text{if}\;n\; \text{is odd}.  
\end{equation}
\end{subequations}
We denote the above collection of coboundary relations by $R$.

Thus, the group of cocycles is a subgroup of $\Psi$ subject to constraints $C$. The group of coboundaries is generated by relations $R$. In other words, we have
\[
  \rH^1(G, \cO_{\rm ev})\cong \frac{\langle (f^{(i)}_J, \phi^{(i)}_J)\in \Psi:  C\rangle}{\langle R\rangle }.  
\]

Next,  from the exponential exact sequence $0\to \mathbb{Z}\to \cO_{\rm ev}\to \cO_{\rm ev}^{*}\to 1$ and the resulting long exact cohomology sequence
\[
  \rH^1(M,\mathbb{Z})\to \rH^1(M,\cO_{\rm ev})\to \rH^1(M,\cO_{\rm ev}^{*})\to \rH^2(M,\mathbb{Z}),
\]
we have \[ \operatorname{Pic}^{0}_{\mathrm{\rm ev}}(M) = \rH^1(M,\cO_{\mathrm{\rm ev}})/\rH^1(M,\mathbb Z), \] so that $\rH^1(M,\mathbb{Z})\cong \rH^1(G,\mathbb{Z}) \cong \mathbb{Z}^{2m}$ can be regarded as a lattice in $\rH^1(M,\cO_{\rm ev})$. 

To understand how this lattice sits inside $\rH^1(G,\cO_{\rm ev})$, we consider an arbitrary element of $\rH^1(G,\mathbb{Z})$ as the cocycle which sends $T_{i}\mapsto m_{i}$ and $S_{i}\mapsto -n_{i}$ where $m_i,n_i\in \mathbb{Z}$ for all $1\leq i\leq m$. Next, we will choose a distinguished representative in each cohomology class so that $S_i\mapsto 0$. In fact, every function $\tilde{F}(z_1,\ldots,z_m,\theta_1,\ldots,\theta_n)= n_{1}z_{1}+n_{2}z_{2}+\dots+n_{m}z_{m} \in \cO_{\rm ev}$ induces a coboundary given by
\begin{align*}
  S_{i} &\mapsto \tilde{F}-S_{i}\tilde{F} = -n_{i},\\
  T_{i} &\mapsto \tilde{F}-T_{i}\tilde{F} = -n_{1}\tau_{i1}-\dots-n_{m}\tau_{im}.  
\end{align*}
As in \cite{Rabin1995}, by subtracting this coboundary from our original cocycle, we get a new representative:
\[
  S_{i}\mapsto 0, \qquad   T_{i}\mapsto m_{i}+n_{1}\tau_{i1}+\dots+n_{m}\tau_{im}.  
\]
Note that all these indeed give us actual 1-cocycles as these are valued in constant functions, hence the constraints are automatically satisfied.

As a result, $\operatorname{Pic}^0_{\rm ev}(M)$ is then given by
\[
\begin{aligned}
\operatorname{Pic}^0_{\rm ev}(M)
&\cong
\frac{\rH^1(G, \cO_{\rm ev})}{\rH^1(G,\mathbb{Z})} \\[6pt]
&\cong
\left(
\frac{
\langle
(f^{(i)}_J,\,
\phi^{(i)}_J) \in \Psi
:\;
J\subseteq \{1,\ldots,n\},\;
1\le i\le m,\;
C
\rangle
}{
\langle R\rangle
}
\right)
\Big/
\mathcal L ,
\end{aligned}
\]
where 
\begin{equation} \label{eq:lattice}
\begin{aligned}
\mathcal{L}
=
\left\{
\begin{aligned}
f^{(i)}_{\emptyset} &= m_i+n_1\tau_{i1}+\ldots+n_m\tau_{im}, \quad m_i,n_i\in \mathbb{Z}, \\ 
f^{(i)}_J &=0, \quad \text{for all}\; J\neq \emptyset, \\
\phi^{(i)}_I &=0, \quad \text{for all}\; I,
\end{aligned}
\right.
\end{aligned}
\end{equation}
which completes the proof.
\end{proof}

\begin{remark}
    Note that $\rH^1(G,\cO_{\rm ev})$ is not a free abelian group due to the relations $R$; this was noted in \cite{BergveltRabin1999} and is why $\rm Pic^0$ is not itself a supermanifold. 
\end{remark}

\begin{remark}
We note that the reduction map $\cO_{\rm ev}^*\rightarrow \cO_{M_{\bos}}^*$ has a section $\cO_{M_{\bos}}^*\hookrightarrow \cO_{\rm ev}^*$ (as $M$ is a projected supermanifold). Thus, the map 
\[
\operatorname{Pic}_{\rm ev}(M)=\rH^1(M, \cO_{\rm ev}^*)\rightarrow \rH^1(M, \cO_{M_{\bos}}^*)=\operatorname{Pic}(M_{\bos})
\] is a split surjection. Since the section $\cO_{M_{\bos}}^*\hookrightarrow \cO_{\rm ev}^*$ is compatible with the exponential exact sequence, we also get a split surjection
\[
\operatorname{Pic}^0_{\rm ev}(M)\rightarrow \operatorname{Pic}^0(M_{\bos}).
\]
Using the description as in Proposition~\ref{prop:picard}, we then get
\[
\begin{aligned}
\operatorname{Pic}^0(M_{\mathrm{\bos}})
&\cong
\frac{
\langle
f^{(i)}_{\emptyset} \in \C
:\;
1\le i\le m
\rangle
}{\langle f^{(i)}_{\emptyset}
=
m_i+n_1\tau_{i1}+\ldots+n_m\tau_{im} \quad
m_i,n_i\in\mathbb Z
\rangle } 
\cong
\frac{\C^m}{G},
\end{aligned}
\]
which recovers the classical case.
\end{remark}

\appendix

\section{Poincar\'e duality tables for $m=2, n=2$} \label{appendixB}
In this appendix, we compute an explicit basis for $\rH^*(M, \cO_M)\cong \rH^*(\cT, E)$ for the supertorus $M$ of dimension $2|2$ and make tables to illustrate Poincar\'e duality in this case.

We fix an ordered basis $\{e_1,e_2\}$ of $\C^2$.
The complex computing $\rH^*(\mathcal{T},E)$ is
\[
\Lambda^0(\C^2)\otimes E\xrightarrow[]{d_0}\Lambda^1(\C^2)\otimes E\xrightarrow[]{d_1}\Lambda^2(\C^2)\otimes E
\] 
where the differentials are given by:
\begin{align*}
d_0(1\otimes x) &=e_1\otimes (T_1-1)x+e_2\otimes (T_2-1)x,\\
d_1(e_1\otimes x + e_2\otimes y) &= (e_1\wedge e_2)\otimes (T_2-1)x - (e_1\wedge e_2)\otimes (T_1-1)y.
\end{align*}
It is isomorphic to the complex
\[
E\xrightarrow[]{\binom{T_1-1}{T_2-1}}E\oplus E\xrightarrow[]{(T_2-1, -(T_1-1))}E.
\]
For any $0\leq i\leq 2$, the Poincar\'e pairing
\[
  \rH^i(\cT, E)\times \rH^{2-i}(\cT, E)\xrightarrow[]{\cup}\rH^2(\cT, E)\xrightarrow[]{\partial_{\theta}}\C
\]
is then induced by \[(\Lambda^i(\C^2)\otimes E) \times (\Lambda^{2-i}(\C^2)\otimes E)\xrightarrow[]{\text{wedge}, \; \text{mult}_{E}}\Lambda^2(\C^2)\otimes E\xrightarrow[]{\partial_{\theta}}\C\] 
where we recall that the map $\partial_{\theta}\colon  \Lambda[\theta_1,\ldots, \theta_n]\rightarrow \C$ extracts the coefficient of the maximal product $\prod_{i,j}{\alpha_{ij}}\prod_{i}\theta_i$ (product taken in lexicographical order) and we identify $\Lambda^2(\C^2)\cong \C$ by sending $e_1\wedge e_2$ to $1$. 

In the next couple of pages, we provide explicit tables exhibiting a basis of $\rH^i$ together with the corresponding dual basis of $\rH^{2-i}$. We will view elements of $\rH^0$ and $\rH^2$ as elements of $E$ and elements of $\rH^1$ as elements of $\Lambda^1(\C^2)\otimes E$. In particular, we note that $\dim \rH^0(\cT,E)=\dim \rH^2(\cT,E)=29$ whereas $\dim \rH^1(\cT,E)=58$.

\begingroup
\small
\setlength{\tabcolsep}{4pt}
\renewcommand{\arraystretch}{1.15}
\refstepcounter{table}
\edef\thistablenumber{\thetable}

\begin{longtable}{|c|p{0.40\linewidth}|p{0.56\linewidth}|}
\hline
\# & $\rH^0$ basis & $\rH^2$ basis \\
\hline
\endfirsthead

\hline
\# & $\rH^0$ basis & $\rH^2$ basis \\
\hline
\endhead

\hline
\endfoot

\multicolumn{3}{c}{}\\[0.6em]
\multicolumn{3}{c}{\parbox{0.9\linewidth}{\centering
\normalsize\textsc{Table \thistablenumber.}\quad
$\rH^0 \longleftrightarrow \rH^2$ Poincar\'e duality pairing for $m=2,n=2$.
}}\\
\endlastfoot

1 & $1$ &
$ \alpha_{11} \alpha_{12} \alpha_{21} \alpha_{22} \theta_{1} \theta_{2}$ \\
\hline

2 & $\alpha_{11}$ &
$ \alpha_{12} \alpha_{21} \alpha_{22} \theta_{1} \theta_{2}$ \\
\hline

3 & $\alpha_{12}$ &
$- \alpha_{11} \alpha_{21} \alpha_{22} \theta_{1} \theta_{2}$ \\
\hline

4 & $\alpha_{21}$ &
$ \alpha_{11} \alpha_{12} \alpha_{22} \theta_{1} \theta_{2}$ \\
\hline

5 & $\alpha_{22}$ &
$- \alpha_{11} \alpha_{12} \alpha_{21} \theta_{1} \theta_{2}$ \\
\hline

6 & $\alpha_{11} \alpha_{12}$ &
$ \alpha_{21} \alpha_{22} \theta_{1} \theta_{2}$ \\
\hline

7 & $\alpha_{11} \alpha_{21}$ &
$- \alpha_{12} \alpha_{22} \theta_{1} \theta_{2}$ \\
\hline

8 & $\alpha_{11} \alpha_{22}$ &
$ \alpha_{12} \alpha_{21} \theta_{1} \theta_{2}$ \\
\hline

9 & $\alpha_{12} \alpha_{21}$ &
$ \alpha_{11} \alpha_{22} \theta_{1} \theta_{2}$ \\
\hline

10 & $\alpha_{12} \alpha_{22}$ &
$- \alpha_{11} \alpha_{21} \theta_{1} \theta_{2}$ \\
\hline

11 & $\alpha_{21} \alpha_{22}$ &
$ \alpha_{11} \alpha_{12} \theta_{1} \theta_{2}$ \\
\hline

12 & $\alpha_{11} \alpha_{12} \alpha_{21}$ &
$ \alpha_{22} \theta_{1} \theta_{2}$ \\
\hline

13 & $\alpha_{11} \alpha_{12} \alpha_{22}$ &
$- \alpha_{21} \theta_{1} \theta_{2}$ \\
\hline

14 & $\alpha_{11} \alpha_{21} \alpha_{22}$ &
$ \alpha_{12} \theta_{1} \theta_{2}$ \\
\hline

15 & $\alpha_{11} \alpha_{21} \theta_{1}$ &
$- \alpha_{12} \alpha_{22} \theta_{2}$ \\
\hline

16 &
$\alpha_{11} \alpha_{21} \theta_{2}
+\alpha_{11} \alpha_{22} \theta_{1}
+\alpha_{12} \alpha_{21} \theta_{1}$ &
$ \alpha_{11} \alpha_{22} \theta_{2}$ \\
\hline

17 &
$\alpha_{11} \alpha_{22} \theta_{2}
+\alpha_{12} \alpha_{21} \theta_{2}
+\alpha_{12} \alpha_{22} \theta_{1}$ &
$- \alpha_{11} \alpha_{21} \theta_{2}$ \\
\hline

18 & $\alpha_{12} \alpha_{21} \alpha_{22}$ &
$- \alpha_{11} \theta_{1} \theta_{2}$ \\
\hline

19 & $\alpha_{12} \alpha_{22} \theta_{2}$ &
$ \alpha_{11} \alpha_{21} \theta_{1}$ \\
\hline

20 & $\alpha_{11} \alpha_{12} \alpha_{21} \alpha_{22}$ &
$ \theta_{1} \theta_{2}$ \\
\hline

21 & $\alpha_{11} \alpha_{12} \alpha_{21} \theta_{1}$ &
$- \alpha_{22} \theta_{2}$ \\
\hline

22 &
$\alpha_{11} \alpha_{12} \alpha_{21} \theta_{2}
+\alpha_{11} \alpha_{12} \alpha_{22} \theta_{1}$ &
$ \alpha_{21} \theta_{2}$ \\
\hline

23 & $\alpha_{11} \alpha_{12} \alpha_{22} \theta_{2}$ &
$- \alpha_{21} \theta_{1}$ \\
\hline

24 & $\alpha_{11} \alpha_{21} \alpha_{22} \theta_{1}$ &
$- \alpha_{12} \theta_{2}$ \\
\hline

25 &
$\alpha_{11} \alpha_{21} \alpha_{22} \theta_{2}
+\alpha_{12} \alpha_{21} \alpha_{22} \theta_{1}$ &
$ \alpha_{11} \theta_{2}$ \\
\hline

26 & $\alpha_{12} \alpha_{21} \alpha_{22} \theta_{2}$ &
$- \alpha_{11} \theta_{1}$ \\
\hline

27 & $\alpha_{11} \alpha_{12} \alpha_{21} \alpha_{22} \theta_{1}$ &
$ \theta_{2}$ \\
\hline

28 & $\alpha_{11} \alpha_{12} \alpha_{21} \alpha_{22} \theta_{2}$ &
$- \theta_{1}$ \\
\hline

29 & $\alpha_{11} \alpha_{12} \alpha_{21} \alpha_{22} \theta_{1} \theta_{2}$ &
$ 1$ \\
\hline
\end{longtable}
\endgroup

\newpage
\begingroup
\small
\setlength{\tabcolsep}{4pt}
\renewcommand{\arraystretch}{1.15}
\newcommand{\mathcell}[1]{\(\begin{array}{@{}l@{}}#1\end{array}\)}
\setcounter{table}{1}
\refstepcounter{table}
\edef\thistablenumber{\thetable}

\begin{longtable}{|c|p{0.40\linewidth}|p{0.56\linewidth}|}
\hline
\# & $\rH^1$ basis & $\rH^1$ basis \\
\hline
\endfirsthead

\hline
\# & $\rH^1$ basis & $\rH^1$ basis \\
\hline
\endhead

\hline
\endfoot

\multicolumn{3}{c}{}\\[0.6em]
\multicolumn{3}{c}{\parbox{0.9\linewidth}{\centering
\normalsize\textsc{Table \thistablenumber.}\quad
$\rH^1 \longleftrightarrow \rH^1$ Poincar\'e duality pairing for $m=2,n=2$.
}}\\
\endlastfoot

1 &
\mathcell{e_{1}\otimes 1} &
\mathcell{e_{2}\otimes \alpha_{11} \alpha_{12} \alpha_{21} \alpha_{22} \theta_{1} \theta_{2}} \\
\hline

2 &
\mathcell{e_{1}\otimes \alpha_{11}} &
\mathcell{e_{1}\otimes \alpha_{11} \alpha_{12} \alpha_{22} \theta_{1} \theta_{2}
+ e_{2}\otimes \alpha_{12} \alpha_{21} \alpha_{22} \theta_{1} \theta_{2}} \\
\hline

3 &
\mathcell{e_{1}\otimes \alpha_{12}} &
\mathcell{-e_{1}\otimes \alpha_{11} \alpha_{12} \alpha_{21} \theta_{1} \theta_{2}
-e_{2}\otimes \alpha_{11} \alpha_{21} \alpha_{22} \theta_{1} \theta_{2}} \\
\hline

4 &
\mathcell{e_{1}\otimes \alpha_{21}} &
\mathcell{e_{2}\otimes \alpha_{11} \alpha_{12} \alpha_{22} \theta_{1} \theta_{2}} \\
\hline

5 &
\mathcell{e_{1}\otimes \alpha_{22}} &
\mathcell{-e_{2}\otimes \alpha_{11} \alpha_{12} \alpha_{21} \theta_{1} \theta_{2}} \\
\hline
6 &
\mathcell{e_{1}\otimes \alpha_{11} \alpha_{12}} &
\mathcell{
-e_{1}\otimes \alpha_{11} \alpha_{12} \alpha_{22} \theta_{1}
-e_{1}\otimes \alpha_{11} \alpha_{22} \theta_{1} \theta_{2} \\
+ e_{1}\otimes \alpha_{12} \alpha_{21} \theta_{1} \theta_{2}
+ e_{2}\otimes \alpha_{21} \alpha_{22} \theta_{1} \theta_{2}
} \\
\hline

7 &
\mathcell{e_{1}\otimes \alpha_{11} \alpha_{22}} &
\mathcell{
e_{1}\otimes \alpha_{11} \alpha_{12} \alpha_{22} \theta_{1}
+ e_{1}\otimes \alpha_{11} \alpha_{12} \theta_{1} \theta_{2} \\
-e_{2}\otimes \alpha_{11} \alpha_{22} \theta_{1} \theta_{2}
+ e_{2}\otimes \alpha_{12} \alpha_{21} \theta_{1} \theta_{2}
} \\
\hline
8 &
\mathcell{e_{1}\otimes \alpha_{11} \theta_{1} -e_{2}\otimes \alpha_{21} \theta_{1}} &
\mathcell{-e_{1}\otimes \alpha_{11} \alpha_{12} \alpha_{22} \theta_{2}} \\
\hline

9 &
\mathcell{e_{1}\otimes \alpha_{11} \theta_{2} -e_{2}\otimes \alpha_{22} \theta_{1}} &
\mathcell{e_{1}\otimes \alpha_{11} \alpha_{12} \alpha_{22} \theta_{1} 
+ e_{2}\otimes \alpha_{12} \alpha_{21} \alpha_{22} \theta_{1}} \\
\hline

10 &
\mathcell{e_{1}\otimes \alpha_{12} \theta_{2} -e_{2}\otimes \alpha_{22} \theta_{2}} &
\mathcell{-e_{1}\otimes \alpha_{11} \alpha_{12} \alpha_{21} \theta_{1}} \\
\hline

11 &
\mathcell{e_{1}\otimes \alpha_{21} \alpha_{22}} &
\mathcell{e_{2}\otimes \alpha_{11} \alpha_{12} \theta_{1} \theta_{2}} \\
\hline

12 &
\mathcell{e_{1}\otimes \alpha_{21} \theta_{1}} &
\mathcell{-e_{2}\otimes \alpha_{11} \alpha_{12} \alpha_{22} \theta_{2}} \\
\hline

13 &
\mathcell{e_{1}\otimes \alpha_{21} \theta_{2} + e_{1}\otimes \alpha_{22} \theta_{1}} &
\mathcell{e_{2}\otimes \alpha_{11} \alpha_{12} \alpha_{21} \theta_{2}} \\
\hline

14 &
\mathcell{e_{1}\otimes \alpha_{22} \theta_{2}} &
\mathcell{-e_{2}\otimes \alpha_{11} \alpha_{12} \alpha_{21} \theta_{1}} \\
\hline

15 &
\mathcell{e_{1}\otimes \alpha_{11} \alpha_{12} \theta_{1}
+ e_{2}\otimes \alpha_{12} \alpha_{21} \theta_{1}} &
\mathcell{e_{1}\otimes \alpha_{12} \alpha_{21} \theta_{2} 
+ e_{2}\otimes \alpha_{21} \alpha_{22} \theta_{2}} \\
\hline

16 &
\mathcell{e_{1}\otimes \alpha_{11} \alpha_{12} \theta_{2}
+ e_{2}\otimes \alpha_{12} \alpha_{22} \theta_{1}} &
\mathcell{e_{1}\otimes \alpha_{11} \alpha_{22} \theta_{1} 
-e_{2}\otimes \alpha_{21} \alpha_{22} \theta_{1}} \\
\hline

17 &
\mathcell{e_{1}\otimes \alpha_{11} \alpha_{21} \theta_{1}} &
\mathcell{-e_{2}\otimes \alpha_{12} \alpha_{22} \theta_{2}} \\
\hline

18 &
\mathcell{e_{1}\otimes \alpha_{11} \alpha_{21} \theta_{2}
+ e_{2}\otimes \alpha_{21} \alpha_{22} \theta_{1}} &
\mathcell{e_{2}\otimes \alpha_{11} \alpha_{22} \theta_{2} 
+ e_{2}\otimes \alpha_{12} \alpha_{22} \theta_{1}} \\
\hline

19 &
\mathcell{e_{1}\otimes \alpha_{11} \alpha_{22} \theta_{1}
-e_{2}\otimes \alpha_{21} \alpha_{22} \theta_{1}} &
\mathcell{e_{1}\otimes \alpha_{11} \alpha_{12} \theta_{2} 
+ e_{2}\otimes \alpha_{12} \alpha_{22} \theta_{1}} \\
\hline

20 &
\mathcell{e_{1}\otimes \alpha_{11} \alpha_{22} \theta_{2}} &
\mathcell{-e_{2}\otimes \alpha_{11} \alpha_{21} \theta_{2}
-e_{2}\otimes \alpha_{12} \alpha_{21} \theta_{1}} \\
\hline

21 &
\mathcell{e_{1}\otimes \alpha_{12} \alpha_{21} \theta_{2}
+ e_{2}\otimes \alpha_{21} \alpha_{22} \theta_{2}} &
\mathcell{e_{1}\otimes \alpha_{11} \alpha_{12} \theta_{1}
-e_{2}\otimes \alpha_{11} \alpha_{21} \theta_{2}} \\
\hline

22 &
\mathcell{e_{1}\otimes \alpha_{12} \alpha_{22} \theta_{2}} &
\mathcell{e_{2}\otimes \alpha_{11} \alpha_{21} \theta_{1}} \\
\hline

23 &
\mathcell{e_{1}\otimes \alpha_{21} \alpha_{22} \theta_{1}} &
\mathcell{e_{2}\otimes \alpha_{11} \alpha_{12} \theta_{2}} \\
\hline

24 &
\mathcell{e_{1}\otimes \alpha_{21} \alpha_{22} \theta_{2}} &
\mathcell{-e_{2}\otimes \alpha_{11} \alpha_{12} \theta_{1}} \\
\hline

25 &
\mathcell{e_{1}\otimes \alpha_{11} \alpha_{12} \alpha_{21} \theta_{1}} &
\mathcell{e_{1}\otimes \alpha_{12} \theta_{2}
-e_{2}\otimes \alpha_{22} \theta_{2}} \\
\hline

26 &
\mathcell{e_{1}\otimes \alpha_{11} \alpha_{12} \alpha_{22} \theta_{1}
+ e_{2}\otimes \alpha_{12} \alpha_{21} \alpha_{22} \theta_{1}} &
\mathcell{e_{1}\otimes \alpha_{11} \alpha_{12} 
-e_{1}\otimes \alpha_{11} \alpha_{22} 
-e_{1}\otimes \alpha_{11} \theta_{2} 
+ e_{2}\otimes \alpha_{22} \theta_{1}} \\
\hline

27 &
\mathcell{e_{1}\otimes \alpha_{11} \alpha_{12} \alpha_{22} \theta_{2}} &
\mathcell{e_{1}\otimes \alpha_{11} \theta_{1} -e_{2}\otimes \alpha_{21} \theta_{1}} \\
\hline

28 &
\mathcell{e_{1}\otimes \alpha_{11} \alpha_{12} \theta_{1} \theta_{2}
-e_{2}\otimes \alpha_{11} \alpha_{22} \theta_{1} \theta_{2} \\
-e_{2}\otimes \alpha_{12} \alpha_{21} \alpha_{22} \theta_{1}
+ e_{2}\otimes \alpha_{12} \alpha_{21} \theta_{1} \theta_{2}} &
\mathcell{-e_{1}\otimes \alpha_{11} \alpha_{22}} \\
\hline

29 &
\mathcell{e_{1}\otimes \alpha_{11} \alpha_{21} \alpha_{22} \theta_{1}} &
\mathcell{-e_{2}\otimes \alpha_{12} \theta_{2}} \\
\hline

30 &
\mathcell{e_{1}\otimes \alpha_{11} \alpha_{21} \alpha_{22} \theta_{2}} &
\mathcell{e_{2}\otimes \alpha_{11} \theta_{2} + e_{2}\otimes \alpha_{12} \theta_{1}} \\
\hline

31 &
\mathcell{e_{1}\otimes \alpha_{11} \alpha_{22} \theta_{1} \theta_{2}
-e_{1}\otimes \alpha_{12} \alpha_{21} \theta_{1} \theta_{2} \\
-e_{2}\otimes \alpha_{12} \alpha_{21} \alpha_{22} \theta_{1}
-e_{2}\otimes \alpha_{21} \alpha_{22} \theta_{1} \theta_{2}} &
\mathcell{e_{1}\otimes \alpha_{11} \alpha_{12}} \\
\hline

32 &
\mathcell{e_{1}\otimes \alpha_{12} \alpha_{21} \alpha_{22} \theta_{2}} &
\mathcell{-e_{2}\otimes \alpha_{11} \theta_{1}} \\
\hline

33 &
\mathcell{e_{1}\otimes \alpha_{21} \alpha_{22} \theta_{1} \theta_{2}} &
\mathcell{e_{2}\otimes \alpha_{11} \alpha_{12}} \\
\hline

34 &
\mathcell{e_{1}\otimes \alpha_{11} \alpha_{12} \alpha_{21} \theta_{1} \theta_{2}
+ e_{2}\otimes \alpha_{11} \alpha_{21} \alpha_{22} \theta_{1} \theta_{2}} &
\mathcell{-e_{1}\otimes \alpha_{12}} \\
\hline

35 &
\mathcell{e_{1}\otimes \alpha_{11} \alpha_{12} \alpha_{22} \theta_{1} \theta_{2}
+ e_{2}\otimes \alpha_{12} \alpha_{21} \alpha_{22} \theta_{1} \theta_{2}} &
\mathcell{e_{1}\otimes \alpha_{11}} \\
\hline

36 &
\mathcell{e_{1}\otimes \alpha_{11} \alpha_{21} \alpha_{22} \theta_{1} \theta_{2}} &
\mathcell{e_{2}\otimes \alpha_{12}} \\
\hline

37 &
\mathcell{e_{1}\otimes \alpha_{12} \alpha_{21} \alpha_{22} \theta_{1} \theta_{2}} &
\mathcell{-e_{2}\otimes \alpha_{11}} \\
\hline

38 &
\mathcell{e_{1}\otimes \alpha_{11} \alpha_{12} \alpha_{21} \alpha_{22} \theta_{1} \theta_{2}} &
\mathcell{e_{2}\otimes 1} \\
\hline

39 &
\mathcell{e_{2}\otimes 1} &
\mathcell{-e_{1}\otimes \alpha_{11} \alpha_{12} \alpha_{21} \alpha_{22} \theta_{1} \theta_{2}} \\
\hline

40 &
\mathcell{e_{2}\otimes \alpha_{11}} &
\mathcell{-e_{1}\otimes \alpha_{12} \alpha_{21} \alpha_{22} \theta_{1} \theta_{2}} \\
\hline

41 &
\mathcell{e_{2}\otimes \alpha_{12}} &
\mathcell{e_{1}\otimes \alpha_{11} \alpha_{21} \alpha_{22} \theta_{1} \theta_{2}} \\
\hline

42 &
\mathcell{e_{2}\otimes \alpha_{11} \alpha_{12}} &
\mathcell{-e_{1}\otimes \alpha_{21} \alpha_{22} \theta_{1} \theta_{2}} \\
\hline

43 &
\mathcell{e_{2}\otimes \alpha_{11} \theta_{1}} &
\mathcell{e_{1}\otimes \alpha_{12} \alpha_{21} \alpha_{22} \theta_{2}} \\
\hline

44 &
\mathcell{e_{2}\otimes \alpha_{11} \theta_{2} + e_{2}\otimes \alpha_{12} \theta_{1}} &
\mathcell{-e_{1}\otimes \alpha_{11} \alpha_{21} \alpha_{22} \theta_{2}} \\
\hline

45 &
\mathcell{e_{2}\otimes \alpha_{12} \theta_{2}} &
\mathcell{e_{1}\otimes \alpha_{11} \alpha_{21} \alpha_{22} \theta_{1}} \\
\hline

46 &
\mathcell{e_{2}\otimes \alpha_{11} \alpha_{12} \theta_{1}} &
\mathcell{-e_{1}\otimes \alpha_{21} \alpha_{22} \theta_{2}} \\
\hline

47 &
\mathcell{e_{2}\otimes \alpha_{11} \alpha_{12} \theta_{2}} &
\mathcell{e_{1}\otimes \alpha_{21} \alpha_{22} \theta_{1}} \\
\hline

48 &
\mathcell{e_{2}\otimes \alpha_{11} \alpha_{21} \theta_{1}} &
\mathcell{e_{1}\otimes \alpha_{12} \alpha_{22} \theta_{2}} \\
\hline

49 &
\mathcell{e_{2}\otimes \alpha_{11} \alpha_{21} \theta_{2}
+ e_{2}\otimes \alpha_{12} \alpha_{21} \theta_{1}} &
\mathcell{-e_{1}\otimes \alpha_{11} \alpha_{22} \theta_{2}
-e_{1}\otimes \alpha_{12} \alpha_{21} \theta_{2} 
-e_{2}\otimes \alpha_{21} \alpha_{22} \theta_{2}} \\
\hline

50 &
\mathcell{e_{2}\otimes \alpha_{11} \alpha_{22} \theta_{2}
+ e_{2}\otimes \alpha_{12} \alpha_{22} \theta_{1}} &
\mathcell{e_{1}\otimes \alpha_{11} \alpha_{21} \theta_{2}
+ e_{2}\otimes \alpha_{21} \alpha_{22} \theta_{1}} \\
\hline

51 &
\mathcell{e_{2}\otimes \alpha_{12} \alpha_{22} \theta_{2}} &
\mathcell{-e_{1}\otimes \alpha_{11} \alpha_{21} \theta_{1}} \\
\hline

52 &
\mathcell{e_{2}\otimes \alpha_{11} \alpha_{12} \alpha_{21} \theta_{1}} &
\mathcell{e_{1}\otimes \alpha_{22} \theta_{2}} \\
\hline

53 &
\mathcell{e_{2}\otimes \alpha_{11} \alpha_{12} \alpha_{21} \theta_{2}} &
\mathcell{-e_{1}\otimes \alpha_{21} \theta_{2}
-e_{1}\otimes \alpha_{22} \theta_{1}} \\
\hline

54 &
\mathcell{e_{2}\otimes \alpha_{11} \alpha_{12} \alpha_{22} \theta_{2}} &
\mathcell{e_{1}\otimes \alpha_{21} \theta_{1}} \\
\hline

55 &
\mathcell{e_{2}\otimes \alpha_{11} \alpha_{12} \theta_{1} \theta_{2}} &
\mathcell{-e_{1}\otimes \alpha_{21} \alpha_{22}} \\
\hline

56 &
\mathcell{e_{2}\otimes \alpha_{11} \alpha_{12} \alpha_{21} \theta_{1} \theta_{2}} &
\mathcell{-e_{1}\otimes \alpha_{22}} \\
\hline

57 &
\mathcell{e_{2}\otimes \alpha_{11} \alpha_{12} \alpha_{22} \theta_{1} \theta_{2}} &
\mathcell{e_{1}\otimes \alpha_{21}} \\
\hline

58 &
\mathcell{e_{2}\otimes \alpha_{11} \alpha_{12} \alpha_{21} \alpha_{22} \theta_{1} \theta_{2}} &
\mathcell{-e_{1}\otimes 1} \\
\hline

\end{longtable}

\endgroup

\bibliography{bibliosuper}

@misc{Noja2026poincaredualitysupergravity,
      title={Poincar\'e Duality and Supergravity}, 
      author={Konstantin Eder and John Huerta and Simone Noja},
      year={2026},
      eprint={2312.05224},
      archivePrefix={arXiv},
      primaryClass={math-ph},
      note ={\arxiv{2312.05224v4}}
}

@book {Shabatcomplexsev92,
    AUTHOR = {Shabat, B. V.},
     TITLE = {Introduction to complex analysis. {P}art {II}},
    SERIES = {Translations of Mathematical Monographs},
    VOLUME = {110},
   EDITION = {Russian},
      NOTE = {Functions of several variables},
 PUBLISHER = {American Mathematical Society, Providence, RI},
      YEAR = {1992},
     PAGES = {x+371},
      ISBN = {0-8218-4611-6},
   MRCLASS = {32-01},
  MRNUMBER = {1192135},
       DOI = {10.1090/mmono/110},
       URL = {https://doi.org/10.1090/mmono/110},
}

@book {GrauertSteinspacebook79,
    AUTHOR = {Grauert, Hans and Remmert, Reinhold},
     TITLE = {Theory of {S}tein spaces},
    SERIES = {Grundlehren der Mathematischen Wissenschaften},
    VOLUME = {236},
      NOTE = {Translated from the German by Alan Huckleberry},
 PUBLISHER = {Springer-Verlag, Berlin-New York},
      YEAR = {1979},
     PAGES = {xxi+249},
      ISBN = {3-540-90388-7},
   MRCLASS = {32-01},
  MRNUMBER = {580152},
}

@book {EnumComb2,
    AUTHOR = {Stanley, Richard P.},
     TITLE = {Enumerative combinatorics. {V}ol. 2},
    SERIES = {Cambridge Studies in Advanced Mathematics},
    VOLUME = {208},
   EDITION = {Second},
      NOTE = {With an appendix by Sergey Fomin},
 PUBLISHER = {Cambridge University Press, Cambridge},
      YEAR = {2024},
     PAGES = {xvi+783},
      ISBN = {978-1-009-26249-1; 978-1-009-26248-4},
   MRCLASS = {05-02 (05A15 05E05 05E10 68R05)},
  MRNUMBER = {4621625},
MRREVIEWER = {Timothy\ Y.\ Chow},
}

@book {Humphreys72Liealgbook,
    AUTHOR = {Humphreys, James E.},
     TITLE = {Introduction to {L}ie algebras and representation theory},
    SERIES = {Graduate Texts in Mathematics},
    VOLUME = {Vol. 9},
 PUBLISHER = {Springer-Verlag, New York-Berlin},
      YEAR = {1972},
     PAGES = {xii+169},
   MRCLASS = {17BXX},
  MRNUMBER = {323842},
MRREVIEWER = {F.\ W.\ Lemire},
}

@article {bwfact,
    AUTHOR = {Sam, Steven V.},
     TITLE = {Borel-{W}eil factorization for super {G}rassmannians},
   JOURNAL = {Selecta Math. (N.S.)},
  FJOURNAL = {Selecta Mathematica. New Series},
    VOLUME = {32},
      YEAR = {2026},
    NUMBER = {2},
     PAGES = {Paper No. 24, 45},
      ISSN = {1022-1824,1420-9020},
   MRCLASS = {14M30 (13D02)},
  MRNUMBER = {5038816},
       DOI = {10.1007/s00029-026-01131-3},
       URL = {https://doi.org/10.1007/s00029-026-01131-3},
}

@article{superres2,
    AUTHOR = {Sam, Steven V and Snowden, Andrew},
     TITLE = {Cohomology of flag supervarieties and resolutions of
              determinantal ideals. {I}{I}},
    note={\arxiv{2412.20797v1}},
    year={2024},
}

@article{kostantrho,
AUTHOR={Sam, Steven V and VandeBogert, Keller and Weyman, Jerzy},
TITLE ={Kostant $\rho$-decomposition of homology {I}. Finite-dimensional representations},
note={\arxiv{2510.01343v1}},
year={2025},
}

@article{expos,
AUTHOR={Sam, Steven V and Snowden, Andrew},
TITLE={Introduction to twisted commutative algebras},
note={\arxiv{1209.5122v1}},
year={2012},
}

@article {superres,
    AUTHOR = {Sam, Steven V and Snowden, Andrew},
     TITLE = {Cohomology of flag supervarieties and resolutions of
              determinantal ideals},
   JOURNAL = {Algebr. Geom.},
  FJOURNAL = {Algebraic Geometry},
    VOLUME = {11},
      YEAR = {2024},
    NUMBER = {1},
     PAGES = {37--70},
      ISSN = {2313-1691,2214-2584},
   MRCLASS = {14M30 (13D02)},
  MRNUMBER = {4680013},
MRREVIEWER = {Anargyros\ Fellouris},
       DOI = {10.14231/ag-2024-002},
       URL = {https://doi.org/10.14231/ag-2024-002},
}

@book {weyman,
    AUTHOR = {Weyman, Jerzy},
     TITLE = {Cohomology of vector bundles and syzygies},
    SERIES = {Cambridge Tracts in Mathematics},
    VOLUME = {149},
 PUBLISHER = {Cambridge University Press, Cambridge},
      YEAR = {2003},
     PAGES = {xiv+371},
      ISBN = {0-521-62197-6},
   MRCLASS = {13D02 (13C40 14L30 14M12 14M15 14M17 20G15)},
  MRNUMBER = {1988690},
MRREVIEWER = {Laurent\ Manivel},
       DOI = {10.1017/CBO9780511546556},
       URL = {https://doi.org/10.1017/CBO9780511546556},
}

@book {FultonHarrisreptheory1991,
    AUTHOR = {Fulton, William and Harris, Joe},
     TITLE = {Representation theory},
    SERIES = {Graduate Texts in Mathematics},
    VOLUME = {129},
      NOTE = {A first course,
              Readings in Mathematics},
 PUBLISHER = {Springer-Verlag, New York},
      YEAR = {1991},
     PAGES = {xvi+551},
      ISBN = {0-387-97527-6; 0-387-97495-4},
   MRCLASS = {20G05 (17B10 20G20 22E46)},
  MRNUMBER = {1153249},
MRREVIEWER = {James\ E.\ Humphreys},
       DOI = {10.1007/978-1-4612-0979-9},
       URL = {https://doi.org/10.1007/978-1-4612-0979-9},
}

@book {GrifiththsHarrisAlgebraicgeo1978,
    AUTHOR = {Griffiths, Phillip and Harris, Joseph},
     TITLE = {Principles of algebraic geometry},
    SERIES = {Pure and Applied Mathematics},
 PUBLISHER = {Wiley-Interscience [John Wiley \& Sons], New York},
      YEAR = {1978},
     PAGES = {xii+813},
      ISBN = {0-471-32792-1},
   MRCLASS = {14-01},
  MRNUMBER = {507725},
MRREVIEWER = {Gerhard\ Pfister},
}

@article {BergveltRabin1999,
    AUTHOR = {Bergvelt, M. J. and Rabin, J. M.},
     TITLE = {Supercurves, their {J}acobians, and super {KP} equations},
   JOURNAL = {Duke Math. J.},
  FJOURNAL = {Duke Mathematical Journal},
    VOLUME = {98},
      YEAR = {1999},
    NUMBER = {1},
     PAGES = {1--57},
      ISSN = {0012-7094,1547-7398},
   MRCLASS = {14H70 (14M30 32C11 37K20 58A50)},
  MRNUMBER = {1683200},
MRREVIEWER = {Daniel\ Hern\'andez Ruip\'erez},
       DOI = {10.1215/S0012-7094-99-09801-0},
       URL = {https://doi.org/10.1215/S0012-7094-99-09801-0},
}

@misc{helein2020introductionsupermanifoldssupersymmetry,
      title={An introduction to supermanifolds and supersymmetry}, 
      author={Frederic Helein},
      year={2020},
      eprint={2006.01870},
      archivePrefix={arXiv},
      primaryClass={math-ph},
      note={\arxiv{2006.01870v1}}, 
}

@book {ManinGraugefieldtheorycomplexgeo,
    AUTHOR = {Manin, Yuri I.},
     TITLE = {Gauge field theory and complex geometry},
    SERIES = {Grundlehren der mathematischen Wissenschaften [Fundamental
              Principles of Mathematical Sciences]},
    VOLUME = {289},
   EDITION = {Second},
      NOTE = {Translated from the 1984 Russian original by N. Koblitz and J.
              R. King,
              With an appendix by Sergei Merkulov},
 PUBLISHER = {Springer-Verlag, Berlin},
      YEAR = {1997},
     PAGES = {xii+346},
      ISBN = {3-540-61378-1},
   MRCLASS = {32-02 (14M30 32C11 32L25 58A50 81-02)},
  MRNUMBER = {1632008},
MRREVIEWER = {Daniel\ Hern\'andez Ruip\'erez},
       DOI = {10.1007/978-3-662-07386-5},
       URL = {https://doi.org/10.1007/978-3-662-07386-5},
}

@article {Witten2019supmanintegration,
    AUTHOR = {Witten, Edward},
     TITLE = {Notes on supermanifolds and integration},
   JOURNAL = {Pure Appl. Math. Q.},
  FJOURNAL = {Pure and Applied Mathematics Quarterly},
    VOLUME = {15},
      YEAR = {2019},
    NUMBER = {1},
     PAGES = {3--56},
      ISSN = {1558-8599,1558-8602},
   MRCLASS = {58A50 (58C50 81T30)},
  MRNUMBER = {3946082},
MRREVIEWER = {Valentin\ Gabriel\ Cristea},
       DOI = {10.4310/PAMQ.2019.v15.n1.a1},
       URL = {https://doi.org/10.4310/PAMQ.2019.v15.n1.a1},
}

@book {Mumfordabelianvarieties1970,
    AUTHOR = {Mumford, David},
     TITLE = {Abelian varieties},
    SERIES = {Tata Institute of Fundamental Research Studies in Mathematics},
    VOLUME = {5},
 PUBLISHER = {Tata Institute of Fundamental Research, Bombay; by Oxford
              University Press, London},
      YEAR = {1970},
     PAGES = {viii+242},
   MRCLASS = {14.51},
  MRNUMBER = {282985},
MRREVIEWER = {James\ Milne},
}

@book {Weibel1994,
    AUTHOR = {Weibel, Charles A.},
     TITLE = {An introduction to homological algebra},
    SERIES = {Cambridge Studies in Advanced Mathematics},
    VOLUME = {38},
 PUBLISHER = {Cambridge University Press, Cambridge},
      YEAR = {1994},
     PAGES = {xiv+450},
      ISBN = {0-521-43500-5; 0-521-55987-1},
   MRCLASS = {18-01 (16-01 17-01 20-01 55Uxx)},
  MRNUMBER = {1269324},
MRREVIEWER = {Kenneth\ A.\ Brown},
       DOI = {10.1017/CBO9781139644136},
       URL = {https://doi.org/10.1017/CBO9781139644136},
}

@article{RabinRhoadesKim2023,
title = {The combinatorics of supertorus sheaf cohomology},
journal = {Journal of Geometry and Physics},
volume = {193},
pages = {104963},
year = {2023},
issn = {0393-0440},
doi = {https://doi.org/10.1016/j.geomphys.2023.104963},
url = {https://www.sciencedirect.com/science/article/pii/S0393044023002152},
author = {Jesse Kim and Jeffrey M. Rabin and Brendon Rhoades}
}

@article {HW87,
    AUTHOR = {Haske, Carl and Wells, Jr., R. O.},
     TITLE = {Serre duality on complex supermanifolds},
   JOURNAL = {Duke Math. J.},
  FJOURNAL = {Duke Mathematical Journal},
    VOLUME = {54},
      YEAR = {1987},
    NUMBER = {2},
     PAGES = {493--500},
      ISSN = {0012-7094,1547-7398},
   MRCLASS = {32L10 (32C37)},
  MRNUMBER = {899403},
MRREVIEWER = {Michael\ G.\ Eastwood},
       DOI = {10.1215/S0012-7094-87-05421-4},
       URL = {https://doi.org/10.1215/S0012-7094-87-05421-4},
}

@article {NojaForms,
    AUTHOR = {Noja, Simone},
     TITLE = {On the geometry of forms on supermanifolds},
   JOURNAL = {Differential Geom. Appl.},
  FJOURNAL = {Differential Geometry and its Applications},
    VOLUME = {88},
      YEAR = {2023},
     PAGES = {Paper No. 101999, 71},
      ISSN = {0926-2245,1872-6984},
   MRCLASS = {58A50 (14F40 32C11)},
  MRNUMBER = {4562762},
       DOI = {10.1016/j.difgeo.2023.101999},
       URL = {https://doi.org/10.1016/j.difgeo.2023.101999},
}

@book {HormanderSCV,
    AUTHOR = {H\"ormander, Lars},
     TITLE = {An introduction to complex analysis in several variables},
    SERIES = {North-Holland Mathematical Library},
    VOLUME = {7},
   EDITION = {Third},
 PUBLISHER = {North-Holland Publishing Co., Amsterdam},
      YEAR = {1990},
     PAGES = {xii+254},
      ISBN = {0-444-88446-7},
   MRCLASS = {32-01 (35N15)},
  MRNUMBER = {1045639},
}

@book {CohofNumberFields,
    AUTHOR = {Neukirch, J\"urgen and Schmidt, Alexander and Wingberg, Kay},
     TITLE = {Cohomology of number fields},
    SERIES = {Grundlehren der mathematischen Wissenschaften [Fundamental
              Principles of Mathematical Sciences]},
    VOLUME = {323},
 PUBLISHER = {Springer-Verlag, Berlin},
      YEAR = {2000},
     PAGES = {xvi+699},
      ISBN = {3-540-66671-0},
   MRCLASS = {11R34 (11-02 11G45 11R23 11S20 11S25 11S31 12G05)},
  MRNUMBER = {1737196},
MRREVIEWER = {Gabriel\ D.\ Villa-Salvador},
}

@book {MacLaneHomology,
    AUTHOR = {Mac Lane, Saunders},
     TITLE = {Homology},
    SERIES = {Classics in Mathematics},
      NOTE = {Reprint of the 1975 edition},
 PUBLISHER = {Springer-Verlag, Berlin},
      YEAR = {1995},
     PAGES = {x+422},
      ISBN = {3-540-58662-8},
   MRCLASS = {18-02 (18Axx 18Cxx 18Gxx)},
  MRNUMBER = {1344215},
}

@article {Rabin1995,
    AUTHOR = {Rabin, Jeffrey M.},
     TITLE = {Super elliptic curves},
   JOURNAL = {J. Geom. Phys.},
  FJOURNAL = {Journal of Geometry and Physics},
    VOLUME = {15},
      YEAR = {1995},
    NUMBER = {3},
     PAGES = {252--280},
      ISSN = {0393-0440,1879-1662},
   MRCLASS = {14M30 (14H52 32C11 58A50 58F07)},
  MRNUMBER = {1316333},
MRREVIEWER = {Alice\ Rogers},
       DOI = {10.1016/0393-0440(94)00012-S},
       URL = {https://doi.org/10.1016/0393-0440(94)00012-S},
}

@article {quotient-supermanifold,
    AUTHOR = {Bartocci, Claudio and Bruzzo, Ugo and Hern\'andez Ruip\'erez,
              Daniel and Pestov, Vladimir},
     TITLE = {Quotient supermanifolds},
   JOURNAL = {Bull. Austral. Math. Soc.},
  FJOURNAL = {Bulletin of the Australian Mathematical Society},
    VOLUME = {58},
      YEAR = {1998},
    NUMBER = {1},
     PAGES = {107--120},
      ISSN = {0004-9727},
   MRCLASS = {58A50},
  MRNUMBER = {1633768},
MRREVIEWER = {Fausto\ Ongay-Larios},
       DOI = {10.1017/S0004972700032044},
       URL = {https://doi.org/10.1017/S0004972700032044},
}

@inproceedings{Noja-Thesis,
  title={TOPICS IN ALGEBRAIC SUPERGEOMETRY OVER PROJECTIVE SPACES},
  author={S. Noja},
  year={2018},
  note={\url{https://api.semanticscholar.org/CorpusID:181539732}}
}

@article {MR2667819,
    AUTHOR = {Alldridge, Alexander and Hilgert, Joachim},
     TITLE = {Invariant {B}erezin integration on homogeneous supermanifolds},
   JOURNAL = {J. Lie Theory},
  FJOURNAL = {Journal of Lie Theory},
    VOLUME = {20},
      YEAR = {2010},
    NUMBER = {1},
     PAGES = {65--91},
      ISSN = {0949-5932},
   MRCLASS = {58C50 (58A50)},
  MRNUMBER = {2667819},
MRREVIEWER = {Frank\ Klinker},
}

@article {MR4544562,
    AUTHOR = {Bruzzo, Ugo and Hern\'andez Ruip\'erez, Daniel and Polishchuk,
              Alexander},
     TITLE = {Notes on fundamental algebraic supergeometry. {H}ilbert and
              {P}icard superschemes},
   JOURNAL = {Adv. Math.},
  FJOURNAL = {Advances in Mathematics},
    VOLUME = {415},
      YEAR = {2023},
     PAGES = {Paper No. 108890, 115},
      ISSN = {0001-8708,1090-2082},
   MRCLASS = {14M30 (14D22 14H10 14K10 83E30)},
  MRNUMBER = {4544562},
MRREVIEWER = {Frans\ Cantrijn},
       DOI = {10.1016/j.aim.2023.108890},
       URL = {https://doi.org/10.1016/j.aim.2023.108890},
}

@article {abhik,
AUTHOR={Pal, Abhik},
TITLE={Cohomology of flag supervarieties and syzygies of compositional varieties},
STATUS={to appear},
YEAR={2026},
}
\end{document}